\documentclass{amsart}

\usepackage{amssymb}
\usepackage{amsmath}
\usepackage{amsthm}
\usepackage{amsfonts}

\usepackage{a4wide}

\newtheorem{theorem}{Theorem}[section]

\newtheorem{corollary}{Corollary}[section]

\theoremstyle{definition}
\newtheorem{definition}{Definition}[section]
\newtheorem{remark}{Remark}[section]
\newtheorem{example}{Example}[section]

\numberwithin{equation}{section}

\begin{document}


\title
[IPADs of \(p\)-tower groups]
{Index-\(p\) abelianization data of \\
\(p\)-class tower groups}

\author{Daniel C. Mayer}
\address{Naglergasse 53\\8010 Graz\\Austria}
\email{algebraic.number.theory@algebra.at}
\urladdr{http://www.algebra.at}

\thanks{Research supported by the Austrian Science Fund (FWF): P 26008-N25}


\date{February 11, 2015}

\begin{abstract}
Given a fixed prime number \(p\),
the multiplet of abelian type invariants
of the \(p\)-class groups of all
unramified cyclic degree \(p\) extensions
of a number field \(K\)
is called its IPAD
(index-\(p\) abelianization data).
These invariants have proved to be a valuable information
for determining the Galois group \(G_p^2\)
of the second Hilbert \(p\)-class field
and the \(p\)-capitulation type \(\varkappa\) of \(K\).
For \(p=3\) and a number field \(K\)
with elementary \(p\)-class group
of rank two, all possible IPADs
are given in the complete form
of several infinite sequences.
Iterated IPADs of second order are used
to identify the group \(G_p^\infty\) of the
maximal unramified pro-\(p\) extension of \(K\).
\end{abstract}

\maketitle



\section{Introduction}
\label{s:Intro}

After a thorough discussion of the terminology used in this article,
the logarithmic and power form of abelian type invariants in \S\
\ref{s:ATI},
and multilayered transfer target types (TTTs),
ordered and accumulated index-\(p\) abelianization data (IPADs)
up to the third order in \S\
\ref{s:IPAD},
we state the main results in \S\
\ref{ss:SporadicIPADs}
on IPADs of exceptional form,
and in \S\
\ref{ss:IPADSequences}
on IPADs in parametrized infinite sequences.
These main theorems give all possible IPADs of number fields \(K\)
with \(3\)-class group \(\mathrm{Cl}_3(K)\) of type \((3,3)\).

Before we turn to applications in extreme computing,
that is, squeezing the computational algebra systems PARI
\cite{PARI}
and MAGMA
\cite{BCP,BCFS,MAGMA}
to their limits in \S\
\ref{s:ExtremeComputing},
where we show how to detect malformed IPADs in \S\
\ref{ss:Sifting},
and how to complete partial \(p\)-capitulation types in \S\
\ref{ss:Completion},
we have to establish a componentwise correspondence
between transfer kernel types (TKTs) and IPADs
in \S\
\ref{s:Correspondence}
by exploiting details of proofs which were given in
\cite{Ma3}.

Iterated IPADs of second order are used in \S\
\ref{s:IPAD2ndOrd}
for the indirect calculation of TKTs in \S\
\ref{ss:Capitulation},
and for determining the exact length \(\ell_p(K)\) of  the \(p\)-class tower
of a number field \(K\) in \S\
\ref{ss:TowerLength}.
This sophisticated technique proves \(\ell_3(K)=3\)
for \(K=\mathbb{Q}(\sqrt{d})\) with
\(d\in\lbrace 342\,664, 957\,013\rbrace\) (the first real quadratic fields)
and \(d=-3\,896\) (the first tough complex quadratic field
after the \lq easy\rq\ \(d=-9\,748\)
\cite{BuMa}),
which resisted all attempts up to now.

Finally, we emphasize that infinite \(p\)-class towers
admit an unknown wealth of possible fine structure in \S\
\ref{s:CmpQdr3Rk3}
on complex quadratic fields \(K\) having a
\(3\)-class group \(\mathrm{Cl}_3(K)\) of type \((3,3,3)\).



\section{Abelian type invariants}
\label{s:ATI}

Let \(p\) be a prime number
and \(A\) be a finite abelian \(p\)-group.
According to the main theorem on finitely generated abelian groups,
there exists a non-negative integer \(r\ge 0\), the \textit{rank} of \(A\),
and a sequence \(n_1,\ldots,n_r\) of positive integers
such that \(n_1\le n_2\le\ldots\le n_r\) and

\begin{equation}
\label{eqn:FinAbGrp1}
A\simeq\mathbb{Z}/p^{n_1}\mathbb{Z}\oplus\ldots\oplus\mathbb{Z}/p^{n_r}\mathbb{Z}.
\end{equation}

\noindent
The powers \(d_i:=p^{n_i}\), \(1\le i\le r\), are known as the \textit{elementary divisors} of \(A\),
since \(d_i\mid d_{i+1}\) for each \(1\le i\le r-1\).
It is convenient to collect equal elementary divisors in formal powers with positive exponents
\(r_1,\ldots ,r_s\) such that \(r_1+\ldots +r_s=r\), \(0\le s\le r\), and
\[n_1=\ldots =n_{r_1}<n_{r_1+1}=\ldots =n_{r_1+r_2}<\ldots <n_{r_1+\ldots +r_{s-1}+1}=\ldots =n_{r_1+\ldots +r_s}.\]
The cumbersome subscripts can be avoided by defining
\(m_j:=n_{r_1+\ldots +r_j}\) for each \(1\le j\le s\).
Then

\begin{equation}
\label{eqn:FinAbGrp2}
A\simeq\left(\mathbb{Z}/p^{m_1}\mathbb{Z}\right)^{r_1}\oplus\ldots\oplus\left(\mathbb{Z}/p^{m_s}\mathbb{Z}\right)^{r_s}
\end{equation}

\noindent
and we can define:



\begin{definition}
\label{dfn:ATI}
The \textit{abelian type invariants} (ATI) of \(A\) are given by the sequence

\begin{equation}
\label{eqn:ATILogForm}
(m_1^{r_1},\ldots,m_s^{r_s})
\end{equation}

\noindent
of strictly increasing positive integers \(m_1<\ldots <m_s\)
with multiplicities \(r_1,\ldots,r_s\) written as formal exponents
indicating iteration.
\end{definition}



\begin{remark}
\label{rmk:ATI}
The integers \(m_j\) are the \(p\)-logarithms
of the elementary divisors \(d_i\).

\begin{enumerate}

\item
For abelian type invariants of high complexity,
the \textit{logarithmic form} in Definition
\ref{dfn:ATI}
requires considerably less space (e.g. in tables) than the usual \textit{power form}

\begin{equation}
\label{eqn:ATIPowForm}
(\overbrace{p^{m_1},\ldots,p^{m_1}}^{r_1},\ldots,\overbrace{p^{m_s},\ldots,p^{m_s}}^{r_s}).
\end{equation}

\item
For brevity, we can even omit the commas separating the entries of the
logarithmic form of abelian type invariants,
provided all the \(m_j\) remain smaller than \(10\).

\item
A further advantage of the brief logarithmic notation is the
independence of the prime \(p\),
in particular when \(p\)-groups with distinct \(p\) are being compared.

\item
Finally, since our preference is to select generators of finite \(p\)-groups
with decreasing orders,
we agree to write abelian type invariants from the right to the left,
in both forms.

\end{enumerate}

\end{remark}

\begin{example}
\label{exm:ATI}
For instance, if \(p=3\), then the abelian type invariants
\((21^4)\) in logarithmic form correspond to the power form \((9,3,3,3,3)\) and
\((2^21^2)\) corresponds to \((9,9,3,3)\).
\end{example}

Now let \(G\) be an arbitrary finite \(p\)-group or infinite topological pro-\(p\) group
with derived subgroup \(G^\prime\) and finite abelianization \(G^{\mathrm{ab}}=G/G^\prime\).



\begin{definition}
\label{dfn:AQI}
The abelian type invariants of the commutator quotient group \(G^{\mathrm{ab}}\)
are called the \textit{abelian quotient invariants} (AQI) of \(G\).
\end{definition}



\section{Index-\(p\) abelianization data}
\label{s:IPAD}

Let \(p\) be a fixed prime number
and \(K\) be a number field with \(p\)-class group \(\mathrm{Cl}_p(K)\)
of order \(p^v\), where \(v\ge 0\) denotes a non-negative integer.

According to the Artin reciprocity law of class field theory
\cite{Ar1},
\(\mathrm{Cl}_p(K)\) is isomorphic to
the commutator quotient group \(G/G^\prime\)
of the Galois group \(G=\mathrm{Gal}(\mathrm{F}_p^\infty(K)\mid K)\)
of the maximal unramified pro-\(p\) extension \(\mathrm{F}_p^\infty(K)\) of \(K\).
\(G\) is called the \(p\)-\textit{tower group} of \(K\).
The fixed field of the commutator subgroup \(G^\prime\) in \(\mathrm{F}_p^\infty(K)\)
is the maximal abelian unramified \(p\)-extension of \(K\),
that is the (first) Hilbert \(p\)-class field \(\mathrm{F}_p^1(K)\) of \(K\)
with Galois group \(\mathrm{Gal}(\mathrm{F}_p^1(K)\mid K)\simeq G/G^\prime\).
The derived subgroup \(G^\prime\) is a closed (and open) subgroup of finite index
\((G:G^\prime)=p^v\) in the topological pro-\(p\) group \(G\).



\begin{definition}
\label{dfn:LayerNT}
For each integer \(0\le n\le v\),
the system

\begin{equation}
\label{eqn:LayerNT}
\mathrm{Lyr}_n(K)=\lbrace K\le L\le\mathrm{F}_p^1(K)\mid\lbrack L:K\rbrack=p^n\rbrace
\end{equation}

\noindent
of intermediate fields \(K\le L\le\mathrm{F}_p^1(K)\)
with relative degree \(\lbrack L:K\rbrack=p^n\)
is called the \(n\)-\textit{th layer} of abelian unramified \(p\)-extensions of \(K\).
In particular,
for \(n=0\), \(K\) forms the \textit{bottom layer}
\(\mathrm{Lyr}_0(K)=\lbrace K\rbrace\),
and for \(n=v\), \(\mathrm{F}_p^1(K)\) forms the \textit{top layer}
\(\mathrm{Lyr}_v(K)=\lbrace\mathrm{F}_p^1(K)\rbrace\).
\end{definition}

Now let \(0\le n\le v\) be a fixed integer
and suppose that \(K\le L\le\mathrm{F}_p^1(K)\) belongs to the \(n\)-th layer.
Then the Galois group \(H=\mathrm{Gal}(\mathrm{F}_p^\infty(K)\mid L)\)
is of finite index \((G:H)=\lbrack L:K\rbrack=p^n\) in the \(p\)-tower group \(G\) of \(K\)
and the quotient \(G/H\simeq\mathrm{Gal}(L\mid K)\) is abelian,
since \(H\) contains the commutator subgroup
\(G^\prime=\mathrm{Gal}(\mathrm{F}_p^\infty(K)\mid\mathrm{F}_p^1(K))\) of \(G\).



\begin{definition}
\label{dfn:LayerGT}
For each integer \(0\le n\le v\),
the system

\begin{equation}
\label{eqn:LayerGT}
\mathrm{Lyr}_n(G)=\lbrace G^\prime\le H\le G\mid (G:H)=p^n\rbrace
\end{equation}

\noindent
of intermediate groups \(G^\prime\le H\le G\)
with index \((G:H)=p^n\)
is called the \(n\)-\textit{th layer} of normal subgroups of \(G\) with abelian quotients \(G/H\).
In particular,
for \(n=0\), \(G\) forms the \textit{top layer}
\(\mathrm{Lyr}_0(G)=\lbrace G\rbrace\),
and for \(n=v\), \(G^\prime\) forms the \textit{bottom layer}
\(\mathrm{Lyr}_v(K)=\lbrace G^\prime\rbrace\).
\end{definition}

A further application of Artin's reciprocity law
\cite{Ar1}
shows that

\begin{equation}
\label{eqn:ArtinReciprocity}
H/H^\prime=\mathrm{Gal}(\mathrm{F}_p^\infty(K)\mid L)/\mathrm{Gal}(\mathrm{F}_p^\infty(K)\mid\mathrm{F}_p^1(L))
\simeq\mathrm{Gal}(\mathrm{F}_p^1(L))\mid L)\simeq\mathrm{Cl}_p(L),
\end{equation}

\noindent
for every subgroup \(H\in\mathrm{Lyr}_n(G)\) and its corresponding extension field \(L\in\mathrm{Lyr}_n(K)\),
where \(0\le n\le v\) is fixed (but arbitrary).

Since the abelianization \(H^{\mathrm{ab}}=H/H^\prime\) forms the target of the
Artin transfer homomorphism \(T_{G,H}:\,G\to H/H^\prime\) from \(G\) to \(H\),
we introduced a preliminary instance of the following terminology in
\cite[Dfn.1.1, p.403]{Ma4}.



\begin{definition}
\label{dfn:TTT}
For each integer \(0\le n\le v\),
the multiplet \(\tau_n(G)=(H/H^\prime)_{H\in\mathrm{Lyr}_n(G)}\),
where each member \(H/H^\prime\) is interpreted rather as its abelian type invariants,
is called the \(n\)-th layer of the \textit{transfer target type} (TTT)
of the pro-\(p\) group \(G\),

\begin{equation}
\label{eqn:TTTGT}
\tau(G)=\lbrack\tau_0(G);\ldots;\tau_v(G)\rbrack,
\text{ where }\tau_n(G)=(H/H^\prime)_{H\in\mathrm{Lyr}_n(G)}
\text{ for each }0\le n\le v.
\end{equation}

\noindent
Similarly, the multiplet \(\tau_n(K)=(\mathrm{Cl}_p(L))_{L\in\mathrm{Lyr}_n(K)}\),
where each member \(\mathrm{Cl}_p(L)\) is interpreted rather as its abelian type invariants,
is called the \(n\)-th layer of the \textit{transfer target type} (TTT)
of the number field \(K\),

\begin{equation}
\label{eqn:TTTNT}
\tau(K)=\lbrack\tau_0(K);\ldots;\tau_v(K)\rbrack,
\text{ where }\tau_n(K)=(\mathrm{Cl}_p(L))_{L\in\mathrm{Lyr}_n(K)}
\text{ for each }0\le n\le v.
\end{equation}

\end{definition}



\begin{remark}
\label{rmk:TTT}

\begin{enumerate}

\item
If it is necessary to specify the underlying prime number \(p\), then
the symbol \(\tau(p,G)\), resp. \(\tau(p,K)\), can  be used for the TTT.

\item
Suppose that \(0<n<v\).
If an ordering is defined for the elements of
\(\mathrm{Lyr}_n(G)\), resp. \(\mathrm{Lyr}_n(K)\),
then the same ordering is applied to the members of
the layer \(\tau_n(G)\), resp. \(\tau_n(K)\),
and the TTT layer is called \textit{ordered}.
Otherwise, the TTT layer is called \textit{unordered}
or \textit{accumulated},
since equal components are collected in powers
with formal exponents denoting iteration.

\item
In view of the considerations in Equation
(\ref{eqn:ArtinReciprocity}),
it is clear that we have the equality

\begin{equation}
\label{eqn:TTT}
\tau(G)=\tau(K),
\end{equation}

\noindent
in the sense of componentwise isomorphisms.

\end{enumerate}

\end{remark}

Since it is increasingly difficult to compute the structure of the \(p\)-class groups \(\mathrm{Cl}_p(L)\)
of extension fields \(L\in\mathrm{Lyr}_n(K)\) in higher layers with \(n\ge 2\),
it is frequently sufficient to make use of information in the first layer only,
that is the layer of subgroups with index \(p\).
Therefore, Boston, Bush and Hajir
\cite{BBH}
invented the following \textit{first order approximation} of the TTT,
a concept which had been used in earlier work already
\cite{BoLG,Bu,BaBu,BoNo,No},
without explicit terminology.



\begin{definition}
\label{dfn:IPAD}
The restriction

\begin{equation}
\label{eqn:IPAD1}
\begin{aligned}
\tau^{(1)}(G) &=\lbrack\tau_0(G);\tau_1(G)\rbrack,\text{ resp.}\\
\tau^{(1)}(K) &=\lbrack\tau_0(K);\tau_1(K)\rbrack,
\end{aligned}
\end{equation}

\noindent
of the TTT \(\tau(G)\), resp. \(\tau(K)\), to the zeroth and first layer
is called the \textit{index-\(p\) abelianization data} (IPAD) of \(G\), resp. \(K\).
\end{definition}

So, the complete TTT is an extension of the IPAD.
However, there also exists another extension of the IPAD
which is not covered by the TTT.
It has also been used already in previous investigations by Boston, Bush and Nover
\cite{Bu,BoNo,No}
and is constructed from the usual IPAD \(\lbrack\tau_0(K);\tau_1(K)\rbrack\) of \(K\),
firstly, by observing that \(\tau_1(K)=(\mathrm{Cl}_p(L))_{L\in\mathrm{Lyr}_1(K)}\)
can be viewed as \(\tau_1(K)=(\tau_0(L))_{L\in\mathrm{Lyr}_1(K)}\) and,
secondly, by extending each \(\tau_0(L)\) to
the IPAD \(\lbrack\tau_0(L);\tau_1(L)\rbrack\) of \(L\).



\begin{definition}
\label{dfn:IteratedIPAD}
The family

\begin{equation}
\label{eqn:IPAD2}
\begin{aligned}
\tau^{(2)}(G) &=\lbrack\tau_0(G);(\lbrack\tau_0(H);\tau_1(H)\rbrack)_{H\in\mathrm{Lyr}_1(G)}\rbrack,\text{ resp.}\\
\tau^{(2)}(K) &=\lbrack\tau_0(K);(\lbrack\tau_0(L);\tau_1(L)\rbrack)_{L\in\mathrm{Lyr}_1(K)}\rbrack,
\end{aligned}
\end{equation}

\noindent
is called the \textit{iterated IPAD of second order} of \(G\), resp. \(K\).
\end{definition}

The concept of iterated IPADs as given in Dfn.
\ref{dfn:IteratedIPAD}
is restricted to the second order and first layers,
and thus is open for further generalization
(higher orders and higher layers).
Since it could be useful for \(2\)-power extensions,
whose absolute degrees increase moderately
and remain manageable by MAGMA or PARI,
we briefly indicate how the \textit{iterated IPAD of third order} could be defined:

\begin{equation}
\label{eqn:IPAD3}
\begin{aligned}
\tau^{(3)}(G) &=\lbrack\tau_0(G);
(\lbrack\tau_0(H);(\lbrack\tau_0(I);\tau_1(I)\rbrack)_{I\in\mathrm{Lyr}_1(H)}\rbrack)_{H\in\mathrm{Lyr}_1(G)}\rbrack,\text{ resp.}\\
\tau^{(3)}(K) &=\lbrack\tau_0(K);
(\lbrack\tau_0(L);(\lbrack\tau_0(M);\tau_1(M)\rbrack)_{M\in\mathrm{Lyr}_1(L)}\rbrack)_{L\in\mathrm{Lyr}_1(K)}\rbrack.
\end{aligned}
\end{equation}



\subsection{Sporadic IPADs}
\label{ss:SporadicIPADs}

In the next two central theorems, we present complete specifications of all possible IPADs
of pro-\(p\) groups \(G\) for \(p=3\) and the simplest case of an abelianization
\(G/G^\prime\) of type \((3,3)\).
We start with pro-\(3\)-groups \(G\) whose metabelianizations \(G/G^{\prime\prime}\)
are vertices on sporadic parts of coclass graphs outside of coclass trees.

Since the abelian type invariants of the members of TTT layers
will depend on the parity of the nilpotency class \(c\) or coclass \(r\),
a more economic notation,
avoiding the tedious distinction of the cases odd or even,
is provided by the following definition.

\begin{definition}
\label{dfn:NearlyHomocyclic}
For an integer \(n\ge 2\),
the \textit{nearly homocyclic abelian \(3\)-group} \(A(3,n)\) of order \(3^n\)
is defined by its type invariants \((q+r,q)\hat{=}(3^{q+r},3^q)\),
where the quotient \(q\ge 1\) and the remainder \(0\le r<2\)
are determined uniquely by the Euclidean division \(n=2q+r\).
Two degenerate cases are included by putting
\(A(3,1)=(1)\hat{=}(3)\) the cyclic group \(C_3\) of order \(3\) and
\(A(3,0)=(0)\hat{=}1\) the trivial group of order \(1\).
\end{definition}

\begin{theorem}
\label{thm:Sporadic3x3}
(First Main Theorem on \(p=3\), \(G/G^\prime\simeq (3,3)\), and \(G/G^{\prime\prime}\) of small class)\\
Let \(G\) be a pro-\(3\) group having a
transfer target type \(\tau(G)=\lbrack\tau_0(G);\tau_1(G);\tau_2(G)\rbrack\)
with top layer component \(\tau_0(G)=1^2\).
Let \(0\le k\le 1\) denote the defect of commutativity
\cite[\S\ 3.1.1, p.412, and \S\ 3.3.2, p.429]{Ma4}
of the metabelianization \(G/G^{\prime\prime}\) of \(G\).
Then the ordered first layer \(\tau_1(G)\) and the bottom layer \(\tau_2(G)\)
are given in the following way.

\begin{enumerate}
\item
If \(G/G^{\prime\prime}\) is of coclass \(\mathrm{cc}(G/G^{\prime\prime})=1\)
and nilpotency class \(\mathrm{cl}(G/G^{\prime\prime})=c\le 3\), then

\begin{equation}
\label{eqn:Cc1TTTLoCl}
\begin{aligned}
\tau_1(G) &=(1)^4;\ \tau_2(G)=(0),\text{ if }c=1,\ G\simeq\langle 9,2\rangle,\\
\tau_1(G) &=(1^2)^4;\ \tau_2(G)=(1),\text{ if }c=2,\ G\simeq\langle 27,3\rangle,\\
\tau_1(G) &=(1^2,\mathbf{(2)^3});\ \tau_2(G)=(1),\text{ if }c=2,\ G\simeq\langle 27,4\rangle,\\
\tau_1(G) &=(\mathbf{1^3},(1^2)^3);\ \tau_2(G)=(1^2),\text{ if }c=3,\ G\simeq\langle 81,7\rangle,\\
\tau_1(G) &=(21,(1^2)^3);\ \tau_2(G)=(1^2),\text{ if }c=3,\ G\simeq\langle 81,8\vert 9\vert 10\rangle,\\
\end{aligned}
\end{equation}

\noindent
where generally \(G^{\prime\prime}=1\).

\item
If \(G/G^{\prime\prime}\) is of coclass \(\mathrm{cc}(G/G^{\prime\prime})=2\)
and nilpotency class \(\mathrm{cl}(G/G^{\prime\prime})=c=3\), then

\begin{equation}
\label{eqn:Cc2TTTCl3}
\begin{aligned}
\tau_1(G) &=((21)^2,1^3,21);\ \tau_2(G)=(1^3),\text{ if }
G\simeq\langle 243,5\rangle\text{ or }G/G^{\prime\prime}\simeq\langle 243,6\rangle,\\
\tau_1(G) &=((21)^2,(1^3)^2);\ \tau_2(G)=(1^3),\text{ if }G/G^{\prime\prime}\simeq\langle 243,3\rangle,\\
\tau_1(G) &=(1^3,21,1^3,21);\ \tau_2(G)=(1^3),\text{ if }G\simeq\langle 243,7\rangle,\\
\tau_1(G) &=(\mathbf{(1^3)^2},21,1^3);\ \tau_2(G)=(1^3),\text{ if }G/G^{\prime\prime}\simeq\langle 243,4\rangle,\\
\tau_1(G) &=(21)^4;\ \tau_2(G)=(1^3),\text{ if }G/G^{\prime\prime}\simeq\langle 243,8\vert 9\rangle,\\
\end{aligned}
\end{equation}

\noindent
where \(G^{\prime\prime}=1\) can be warranted for \(G/G^{\prime\prime}\simeq\langle 243,5\vert 7\rangle\) only.

However, if \(\mathrm{cl}(G/G^{\prime\prime})=c=4\) with \(k=1\), then

\begin{equation}
\label{eqn:Cc2TTTCl4}
\begin{aligned}
\tau_1(G) &=((21)^2,(1^3)^2);\ \tau_2(G)=(21^2),\text{ if }G/G^{\prime\prime}\simeq\langle 729,37\vert 38\vert 39\rangle,\\
\tau_1(G) &=((21)^2,(1^3)^2);\ \tau_2(G)=\mathbf{(1^4)},\text{ if }G/G^{\prime\prime}\simeq\langle 729,34\vert 35\vert 36\rangle,\\
\tau_1(G) &=((1^3)^2,21,1^3);\ \tau_2(G)=(21^2),\text{ if }G/G^{\prime\prime}\simeq\langle 729,44\vert 45\vert 46\vert 47\rangle,\\
\tau_1(G) &=(21)^4;\ \tau_2(G)=\mathbf{(1^4)},\text{ if }G/G^{\prime\prime}\simeq\langle 729,56\vert 57\rangle.\\
\end{aligned}
\end{equation}

\item
If \(G/G^{\prime\prime}\) is of coclass \(\mathrm{cc}(G/G^{\prime\prime})=r\ge 3\)
and nilpotency class \(\mathrm{cl}(G/G^{\prime\prime})=c=r+1\), then

\begin{equation}
\label{eqn:HiCcTTTClMin}
\tau_1(G)=(A(3,r+1)^2,(1^2)^3);\ \tau_2(G)=A(3,r)\times A(3,r-1)\text{ and }k=0.
\end{equation}

\noindent
However, if \(c=r+2\), then

\begin{equation}
\label{eqn:HiCcTTTCrit}
\begin{aligned}
\tau_1(G) &=(A(3,r+2),A(3,r+1),(1^2)^3);\ \tau_2(G)=A(3,r+1)\times A(3,r-1),\text{ if }k=0\\
\tau_1(G) &=(A(3,r+1)^2,(1^2)^3);\ \tau_2(G)=A(3,r+1)\times A(3,r-1),\text{ if }k=1,\text{ regular case},\\
\tau_1(G) &=(A(3,r+1)^2,(1^2)^3);\ \tau_2(G)=\mathbf{A(3,r)\times A(3,r)},\text{ if }k=1,\text{ irregular case},\\
\end{aligned}
\end{equation}

\noindent
where the irregular case can only occur for even class and coclass \(c=r+2\equiv 0\pmod{2}\),
positive defect of commutativity \(k=1\), and relational parameter \(\rho=-1\) in
\cite[Eqn.(3.6), p.424]{Ma3}
or
\cite[Eqn.(3.3), p.430]{Ma4}.

\end{enumerate}

\end{theorem}






\begin{proof}

Since this proof heavily relies on our earlier paper
\cite{Ma3},
it should be pointed out that, for a \(p\)-group \(G\),
the index of nilpotency \(m=c+1\) is used generally instead of the
nilpotency class \(\mathrm{cl}(G)=c=m-1\)
and the invariant \(e=r+1\) frequently (but not always) replaces the
coclass \(\mathrm{cc}(G)=r=e-1\)
in that paper.

\begin{enumerate}

\item
Using the association between the identifier of \(G\) in the SmallGroups Library
\cite{BEO1,BEO2}
and the transfer kernel type (TKT)
\cite{Ma2},
which is visualized in
\cite[Fig.3.1, p.423]{Ma3},
this claim follows from
\cite[Thm.4.1, p.427, and Tbl.4.1, p.429]{Ma3}.

\item
For \(c=3\), resp. \(c=4\) with \(k=1\), the statement is a consequence of
\cite[Thm.4.2 and Tbl.4.3, p.434]{Ma3},
resp.
\cite[Thm.4.3 and Tbl.4.5, p.438]{Ma3},
when the association between the identifier of \(G\) in the SmallGroups Database
and the TKT is taken into consideration,
as visualized in
\cite[Fig.4.1, p.433]{Ma3}.

\item
All the regular cases behave completely similar as the general case in Theorem
\ref{thm:Sequences3x3}, item (3), Equation
(\ref{eqn:HiCcTTT}).
In the irregular case, only the bottom layer \(\tau_2(G)\),
consisting of the abelian quotient invariants \(G^\prime/G^{\prime\prime}\) of the derived subgroup \(G^\prime\),
is exceptional and must be taken from
\cite[Appendix \S\ 8, Thm.8.8, p.461]{Ma3}.
\end{enumerate}

\end{proof}



\subsection{Infinite IPAD sequences}
\label{ss:IPADSequences}

Now we come to the IPADs of pro-\(p\)-groups \(G\)
whose metabelianizations \(G/G^{\prime\prime}\) are members of infinite periodic sequences,
inclusively mainlines, of coclass trees.



\begin{theorem}
\label{thm:Sequences3x3}
(Second Main Theorem on \(p=3\), \(G/G^\prime\simeq (3,3)\), and \(G/G^{\prime\prime}\) of large class)\\
Let \(G\) be a pro-\(3\) group having a
transfer target type \(\tau(G)=\lbrack\tau_0(G);\tau_1(G);\tau_2(G)\rbrack\)
with top layer component \(\tau_0(G)=1^2\).
Let \(0\le k\le 1\) denote the defect of commutativity
\cite[\S\ 3.1.1, p.412, and \S\ 3.3.2, p.429]{Ma4}
of the metabelianization \(G/G^{\prime\prime}\) of \(G\).
Then the ordered first layer \(\tau_1(G)\) and the bottom layer \(\tau_2(G)\)
are given in the following way.

\begin{enumerate}
\item
If \(G/G^{\prime\prime}\) is of coclass \(\mathrm{cc}(G/G^{\prime\prime})=1\)
and nilpotency class \(\mathrm{cl}(G/G^{\prime\prime})=c\ge 4\), then

\begin{equation}
\label{eqn:Cc1TTT}
\begin{aligned}
\tau_1(G) &=(A(3,c-k),(1^2)^3); \\
\tau_2(G) &=A(3,c-1).
\end{aligned}
\end{equation}

\item
If \(G/G^{\prime\prime}\) is of coclass \(\mathrm{cc}(G/G^{\prime\prime})=2\)
and nilpotency class \(\mathrm{cl}(G/G^{\prime\prime})=c\ge 5\),
or \(c=4\) with \(k=0\), then

\begin{equation}
\label{eqn:Cc2TTT1}
\begin{aligned}
\tau_1(G) &=(A(3,c-k),21,(1^3)^2)\text{ or} \\
\tau_1(G) &=(A(3,c-k),21,1^3,21)\text{ or} \\
\tau_1(G) &=(A(3,c-k),(21)^3), \\
\end{aligned}
\end{equation}

\noindent
in dependence on the coclass tree
\(G/G^{\prime\prime}\in\mathcal{T}^2(\langle 729,i\rangle)\), \(i\in\lbrace 40,49,54\rbrace\),
but uniformly

\begin{equation}
\label{eqn:Cc2TTT2}
\tau_2(G)=A(3,c-1)\times A(3,1).
\end{equation}

\item
If \(G/G^{\prime\prime}\) is of coclass \(\mathrm{cc}(G/G^{\prime\prime})=r\ge 3\)
and nilpotency class \(\mathrm{cl}(G/G^{\prime\prime})=c\ge r+3\),
or \(c=r+2\) with \(k=0\), then

\begin{equation}
\label{eqn:HiCcTTT}
\begin{aligned}
\tau_1(G) &=(A(3,c-k),A(3,r+1),(1^2)^3); \\
\tau_2(G) &=A(3,c-1)\times A(3,r-1).
\end{aligned}
\end{equation}

\end{enumerate}

The first member \(H_1/H_1^\prime\) of the ordered first layer \(\tau_1(G)\)
reveals a uni-polarization (dependence on the nilpotency class \(c\))
whereas the other three members \(H_i/H_i^\prime\), \(2\le i\le 4\),
show a stabilization (independence of \(c\)) for fixed coclass \(r\).

\end{theorem}



\begin{proof}

Again, we make use of
\cite{Ma3},
and we point out that, for a \(p\)-group \(G\),
the index of nilpotency \(m=c+1\) is used generally instead of the
nilpotency class \(\mathrm{cl}(G)=c=m-1\)
and the invariant \(e=r+1\) frequently (but not always) replaces the
coclass \(\mathrm{cc}(G)=r=e-1\)
in that paper.

\begin{enumerate}

\item
All components of \(\tau_1(G)\) are given in
\cite[\S\ 3.1, Thm.3.1, Eqn.(3.4)--(3.5), p.421]{Ma3}
when their ordering is defined by the special selection of generators
\cite[\S\ 3.1, Eqn.(3.1)--(3.2), p.420]{Ma3}.
There is only a unique coclass tree with \(3\)-groups of coclass \(1\).

\item
The first component of \(\tau_1(G)\) is given in
\cite[\S\ 3.2, Thm.3.2, Eqn.(3.7), p.424]{Ma3},
and the last three components of \(\tau_1(G)\) are given in
\cite[\S\ 4.5, Thm.4.4, p.440]{Ma3}
and
\cite[\S\ 4.5, Tbl.4.7, p.441]{Ma3},
when their ordering is defined by the special selection of generators
\cite[\S\ 3.2, Eqn.(3.6), p.424]{Ma3}.
The invariant \(\varepsilon\in\lbrace 0,1,2\rbrace\)
\cite{Ma3},
which counts IPAD components of rank \(3\),
decides to which of the mentioned three coclass trees the group \(G\) belongs
\cite[Fig.3.6--3.7, pp.442--443]{Ma4}.

\item
The first two components of \(\tau_1(G)\) are given in
\cite[\S\ 3.2, Thm.3.2, Eqn.(3.7)--(3.8), p.424]{Ma3},
and the last two components of \(\tau_1(G)\) are given in
\cite[\S\ 4.6, Thm.4.5, p.444]{Ma3},
when their ordering is defined by the special selection of generators
\cite[\S\ 3.2, Eqn.(3.6), p.424]{Ma3}.
For coclass bigger than \(2\), it is irrelevant
to which of the four (in the case of odd coclass \(r\))
or six (in the case of even coclass \(r\))
coclass trees the group \(G\) belongs.
The IPAD is independent of this detailed information,
provided that \(c\ge r+3\).
\end{enumerate}

Finally, the bottom layer \(\tau_2(G)\),
consisting of the abelian quotient invariants \(G^\prime/G^{\prime\prime}\) of the derived subgroup \(G^\prime\),
is generally taken from
\cite[Appendix \S\ 8, Thm.8.8, p.461]{Ma3}.

\end{proof}





\section{Componentwise correspondence of IPAD and TKT}
\label{s:Correspondence}

Within this section, where generally \(p=3\), we employ some special terminology.
We say a class of a base field \(K\) remains \textit{resistant} if
it does not capitulate in any unramified cyclic cubic extension \(L\vert K\).
When the \(3\)-class group of \(K\) is of type \((3,3)\)
the next layer of unramified abelian extensions is already the top layer
consisting of the Hilbert \(3\)-class field \(\mathrm{F}_3^1(K)\),
where the resistant class must capitulate,
according to the Hilbert/Artin/Furtw\"angler principal ideal theorem.

Our desire is to show that the components of the ordered IPAD and TKT
are in a strict correspondence to each other.
For this purpose, we use details of the proofs given in
\cite{Ma3},
where generators of metabelian \(3\)-groups \(G\) with \(G/G^\prime\simeq (3,3)\)
were selected in a canonical way, particularly adequate for theoretical aspects.
Since we now prefer a more computational aspect,
we translate the results into a form which is given by the
computational algebra system MAGMA
\cite{MAGMA}.

To be specific, we choose the vertices of two important coclass trees
for illustrating these peculiar techniques.
The vertices of depth (distance from the mainline) at most \(1\)
of both coclass trees,
with roots \(\langle 243,6\rangle\) and \(\langle 243,8\rangle\)
\cite[Fig.3.6--3.7, pp.442--443]{Ma4},
are metabelian \(3\)-groups \(G\) with
order \(\lvert G\rvert\ge 3^5\),
nilpotency class \(c=\mathrm{cl}(G)\ge 3\), and
fixed coclass \(\mathrm{cc}(G)=2\).



\subsection{The coclass tree \(\mathcal{T}^2(\langle 243,6\rangle)\)}
\label{ss:Ord243Id6}

\begin{remark}
\label{rmk:Ord243Id6}
The first layers of the TTT and TKT of vertices of depth at most \(1\)
of the coclass tree \(\mathcal{T}^2(\langle 243,6\rangle)\)
\cite[Fig.3.6, p.442]{Ma4}
consist of four components each, and
share the following common properties
with respect to MAGMA's selection of generators:

\begin{enumerate}
\item
polarization (dependence on the class \(c\)) at the first component,
\item
stabilization (independence of the class \(c\)) at the last three components,
\item
rank \(3\) at the second TTT component (\(\varepsilon=1\) in
\cite{Ma3}).
\end{enumerate}

Using the class \(c\), resp. an asterisk, as wildcard characters,
these common properties can be summarized as follows,
now including the details of the stabilization:

\begin{equation}
\label{eqn:Ord243Id6}
\tau_1(G)=\lbrack A(3,c),1^3,(21)^2\rbrack,\text{ and } \varkappa_1(G)=(\ast,1,2,2).
\end{equation}

\end{remark}

However, to assure the general applicability of the theorems and corollaries in this section,
we aim at independency of the selection of generators
(and thus invariance under permutations).

\begin{theorem}
\label{thm:Ord243Id6}
(in field theoretic terminology)
\begin{enumerate}
\item
The class associated with the polarization becomes principal in the extension with rank \(3\).
\item
The class associated with rank \(3\) becomes principal in both extensions of type \((21)\),
in particular, \(\varkappa_1(G)\) cannot be a permutation and can have at most one fixed point.
\end{enumerate}
\end{theorem}

\begin{remark}
\label{rmk:Ord243Id6Var}
Aside from the common properties, there also arise variations due to the polarization,
which we first express with respect to MAGMA's selection of generators:

\begin{enumerate}
\item
The TKT is E.6, \(\varkappa_1(G)=(1,1,2,2)\), if and only if
the polarized extension reveals a \textit{fixed point} principalization.
\item
The TKT is E.14, \(\varkappa_1(G)\in\lbrace (3,1,2,2),(4,1,2,2)\rbrace\), if and only if
one of the classes associated with type \((21)\) becomes principal in the polarized extension.
\item
The TKT is H.4, \(\varkappa_1(G)=(2,1,2,2)\), if and only if
the class associated with rank \(3\) becomes principal in the polarized extension.
\item
The TKT is c.18, \(\varkappa_1(G)=(0,1,2,2)\), if and only if
the polarized extension reveals a \textit{total} principalization (indicated by \(0\)).
\end{enumerate}

\end{remark}

\begin{corollary}
\label{cor:Ord243Id6}
(in field theoretic terminology)
\begin{enumerate}
\item
For the TKTs E.6 and H.4, both classes associated with type \((21)\) remain resistant,
for TKT E.14 only one of them.
\item
All extensions satisfy Taussky's condition (B)
\cite{Ta},
with the single exception of of the polarized extension in the case of TKT E.6 or c.18,
which satisfies condition (A).
\item
TKT E.6 has a single fixed point,
E.14 contains a \(3\)-cycle, and
H.4 contains a \(2\)-cycle.
\end{enumerate}
\end{corollary}

\begin{proof}
(of Theorem
\ref{thm:Ord243Id6}
and Corollary
\ref{cor:Ord243Id6})\\
Observe that in
\cite{Ma3},
the index of nilpotency \(m=c+1\) and the invariant \(e=r+1\)
are used rather than the nilpotency class \(c=m-1\) and the coclass \(r=e-1\).
The claims are a consequence of
\cite[\S\ 4.5, Tbl.4.7, p.441]{Ma3},
when we perform a permutation from the first layer TKT and TTT
\[\varkappa_1(G)=(\ast,3,1,3),\ \tau_1(G)=\lbrack A(3,c),21,1^3,21\rbrack,\]
with respect to the canonical generators, to the corresponding invariants
\[\varkappa_1(G)=(\ast,1,2,2),\ \tau_1(G)=\lbrack A(3,c),1^3,(21)^2\rbrack,\]
with respect to MAGMA's generators.
\end{proof}



\subsection{The coclass tree \(\mathcal{T}^2(\langle 243,8\rangle)\)}
\label{ss:Ord243Id8}

\begin{remark}
\label{rmk:Ord243Id8}
The first layer TTT and TKT of vertices of depth at most \(1\)
of the coclass tree \(\mathcal{T}^2(\langle 243,8\rangle)\)
\cite[Fig.3.7, p.443]{Ma4}
consist of four components each, and
share the following common properties
with respect to MAGMA's choice of generators:

\begin{enumerate}
\item
polarization (dependence on the class \(c\)) at the second component,
\item
stabilization (independence of the class \(c\)) at the other three components,
\item
rank \(3\) does not occur at any TTT component (\(\varepsilon=0\) in
\cite{Ma3}).
\end{enumerate}

Using the class \(c\), resp. an asterisk, as wildcard characters,
the common properties can be summarized as follows,
now including details of the stabilization:

\begin{equation}
\label{eqn:Ord243Id8}
\tau_1(G)=\lbrack 21,A(3,c),(21)^2\rbrack,\text{ and } \varkappa_1(G)=(2,\ast,3,4).
\end{equation}

\end{remark}

Again, we have to ensure the general applicability of the following theorem and corollary,
which must be independent of the choice of generators
(and thus invariant under permutations).

\begin{theorem}
\label{thm:Ord243Id8}
(in field theoretic terminology)
\begin{enumerate}
\item
Two extensions of type \((21)\) reveal fixed point principalization
satisfying condition (A)
\cite{Ta}.
\item
The remaining extension of type \((21)\) satisfies condition (B),
since the class associated with the polarization becomes principal there.
\end{enumerate}
\end{theorem}

\begin{remark}
\label{rmk:Ord243Id8Var}
Next, we come to variations caused by the polarization,
which we now express with respect to MAGMA's choice of generators:

\begin{enumerate}
\item
The TKT is E.8, \(\varkappa_1(G)=(2,2,3,4)\), if and only if
the polarized extension reveals a \textit{fixed point} principalization.
\item
The TKT is E.9, \(\varkappa_1(G)\in\lbrace (2,3,3,4),(2,4,3,4)\rbrace\), if and only if
one of the classes associated with fixed points becomes principal in the polarized extension.
\item
The TKT is G.16, \(\varkappa_1(G)=(2,1,3,4)\), if and only if
the class associated with type \((21)\), satisfying condition (B), becomes principal in the polarized extension.
\item
The TKT is c.21, \(\varkappa_1(G)=(2,0,3,4)\), if and only if
the polarized extension reveals a \textit{total} principalization (indicated by \(0\)).
\end{enumerate}

\end{remark}

\begin{corollary}
\label{cor:Ord243Id8}
(in field theoretic terminology)
\begin{enumerate}
\item
For the TKTs E.8 and E.9, the class associated with the polarization remains resistant,
\item
The polarized extension satisfies condition (B)
\cite{Ta}
in the case of TKT E.9 or G.16,
and it satisfies condition (A) in the case of TKT E.8 or c.21.
\item
TKT G.16 is a permutation containing a \(2\)-cycle, and
TKT E.8 is the unique TKT possessing three fixed points.
\end{enumerate}
\end{corollary}

\begin{proof}
(of Theorem
\ref{thm:Ord243Id8}
and Corollary
\ref{cor:Ord243Id8})\\
In our paper
\cite{Ma3},
the index of nilpotency \(m=c+1\) and the invariant \(e=r+1\)
are used rather than the nilpotency class \(c=m-1\) and the coclass \(r=e-1\).
All claims are a consequence of
\cite[\S\ 4.5, Tbl.4.7, p.441]{Ma3},
provided we perform a transformation from the first layer TKT and TTT
\[\varkappa_1(G)=(\ast,2,3,1),\ \tau_1(G)=\lbrack A(3,c),(21)^3\rbrack,\]
with respect to the canonical generators, to the corresponding invariants
\[\varkappa_1(G)=(2,\ast,3,4),\ \tau_1(G)=\lbrack 21,A(3,c),(21)^2\rbrack,\]
with respect to MAGMA's generators.
\end{proof}



\section{Applications in extreme computing}
\label{s:ExtremeComputing}



\subsection{Application 1: Sifting malformed IPADs}
\label{ss:Sifting}
 
\begin{definition}
\label{dfn:Malformed}
An IPAD with bottom layer component
\(\tau_0(K)=(3,3)\)
is called \textit{malformed}
if it is not covered by Theorems
\ref{thm:Sporadic3x3}
and
\ref{thm:Sequences3x3}.
\end{definition}

To verify predicted asymptotic densities of maximal unramified pro-\(3\) extensions
in the article
\cite{BBH}
numerically,
the IPADs of all complex quadratic fields \(K=\mathbb{Q}(\sqrt{d})\)
with discriminants \(-10^8<d<0\) and \(3\)-class rank \(r_3(K)=2\) were computed
with the aid of PARI/GP
\cite{PARI}.
In particular, there occurred \(276\,375\), resp. \(122\,444\), such fields
with \(3\)-class group \(\mathrm{Cl}_3(K)\) of type \((3,3)\), resp. \((9,3)\).



\begin{example}
\label{exm:Malformed3x3}

A check of all \(276\,375\) IPADs for complex quadratic fields with type \((3,3)\)
in the range \(-10^8<d<0\) of discriminants,
for which Theorem
\ref{thm:Sequences3x3}
states that
the \(3\)-class groups of the \(4\) unramified cyclic cubic extensions
can only have \(3\)-rank \(2\), except for the unique type \((3,3,3)\),
revealed that the following \(5\) IPADs were computed erroneously
by the used version of PARI/GP
\cite{PARI}
in
\cite{BBH}.
The successful recomputation was done with MAGMA
\cite{MAGMA}.

\begin{enumerate}

\item
For \(d=-96\,174\,803\), the erroneous IPAD
\(\tau^{(1)}(K)=\lbrack (3,3);(3,3,3),(9,3,3,3),(27,9)^2\rbrack\)
contained the malformed component \((9,3,3,3)\) instead of the correct \((3,3,3)\).
The transfer kernel type (TKT)
\cite{Ma2,Ma4}
turned out to be F.12.

\item
For \(d=-77\,254\,244\), the erroneous IPAD
\(\tau^{(1)}(K)=\lbrack (3,3);(3,3,3)^2,(3,3,3,3),(9,3)\rbrack\)
contained the malformed component \((3,3,3,3)\) instead of the correct \((3,3,3)\).
Its TKT is H.4.

\item
For \(d=-73\,847\,683\), the erroneous IPAD
\(\tau^{(1)}(K)=\lbrack (3,3);(3,3,3),(9,3,3),(9,3)^2\rbrack\)
contained the malformed component \((9,3,3)\) instead of the correct \((9,3)\).
The TKT is D.10.

\item
For \(d=-81\,412\,223\), the erroneous IPAD
\(\tau^{(1)}(K)=\lbrack (3,3);(9,3,3),(9,3)^2,(27,9)\rbrack\)
contained the malformed component \((9,3,3)\) instead of the correct \((9,3)\).
This could be a TKT E.8 or E.9 or G.16.

\item
For \(d=-82\,300\,871\), the erroneous IPAD
\(\tau^{(1)}(K)=\lbrack (3,3);(3,3,3),(9,3),(9,9,3),(27,9)\rbrack\)
contained the malformed component \((9,9,3)\) instead of the correct \((9,3)\).
This could be a TKT E.6 or E.14 or H.4.

\end{enumerate}

For the last two cases, Magma failed to determine the TKT.
Nevertheless, none of the discriminants
\[d\in\lbrace -73\,847\,683, -77\,254\,244, -81\,412\,223, -82\,300\,871, -96\,174\,803\rbrace\]
is particularly spectacular.
 
\end{example}



\begin{example}
\label{exm:Malformed9x3}

We also checked all \(122\,444\) IPADs for complex quadratic fields with type \((9,3)\)
in the range \(-10^8<d<0\) of discriminants,
Again, we found exactly \(5\) errors
among these IPADs which had been computed by PARI/GP
\cite{PARI}
in
\cite{BBH}.
For the recomputation we used MAGMA
\cite{MAGMA}.
The study of this extensive material
was very helpful for the deeper understanding
of \(3\)-groups having abelianization of type \((9,3)\).
Systematic results in the style of Theorems
\ref{thm:Sporadic3x3}
and
\ref{thm:Sequences3x3}
will be given in a forthcoming paper.
The abbreviation pTKT means the \textit{punctured} TKT.

\begin{enumerate}

\item
For \(d=-94\,304\,231\), the erroneous IPAD
\(\tau^{(1)}(K)=\lbrack (9,3);(9,3,3),(27,3),(9,9,9),(27,9)\rbrack\)
contained the malformed component \((27,9)\) instead of the correct \((27,3)\).
This could be a homocyclic pTKT B.2 or C.4 or D.5.
 
\item
For \(d=-79\,749\,087\), the erroneous IPAD
\(\tau^{(1)}(K)=\lbrack (9,3);(9,3,3)^2,(27,3,3),(27,3)\rbrack\)
contained the malformed component \((27,3,3)\) instead of the correct \((27,3)\).
It is a pTKT D.11.

\item 
For \(d=-74\,771\,240\), the erroneous IPAD
\(\tau^{(1)}(K)=\lbrack (9,3);(9,3,3),(27,3,3),(27,3),(9,9,9)\rbrack\)
contained the malformed component \((27,3,3)\) instead of the correct \((27,3)\).
It could be a homocyclic pTKT B.2 or C.4 or D.5.

\item 
For \(d=-70\,204\,919\), the erroneous IPAD
\(\tau^{(1)}(K)=\lbrack (9,3);(9,3,3),(27,3)^2,(81,27,27)\rbrack\)
contained the malformed component \((81,27,27)\) instead of the correct \((81,27,3)\).
This could be a pTKT B.2 or C.4 or D.5 in the first excited state.

\item 
For \(d=-86\,139\,199\), the erroneous IPAD
\(\tau^{(1)}(K)=\lbrack (9,3);(81,3,3,3),(9,3,3),(27,3)^2\rbrack\)
contained the malformed component \((81,3,3,3)\) instead of the correct \((9,3,3)\).
This is clearly a pTKT D.11.

\end{enumerate}

Again, none of the corresponding discriminants
\[d\in\lbrace -70\,204\,919, -74\,771\,240, -79\,749\,087, -86\,139\,199, -94\,304\,231\rbrace\]
is particularly spectacular.

\end{example}

We emphasize that,
in both Examples
\ref{exm:Malformed3x3}
and
\ref{exm:Malformed9x3},
the errors of PARI/GP
\cite{PARI}
occured in the upper limit range
of absolute discriminants above \(70\) millions.
This seems to be a critical region of extreme computing where
current computational algebra systems become unstable. 
MAGMA
\cite{MAGMA}
also often fails to compute the TKT in that range.

Fortunately, there appeared a single discriminant only
for each of the \(5\) erroneous IPADs,
in both examples.
This indicates that the errors are not systematic
but rather stochastic.



\subsection{Application 2: Completing partial capitulation types}
\label{ss:Completion}

\begin{example}
\label{exm:TKTH4}
For the discriminant \(d=-3\,849\,267\) of
the complex quadratic field \(K=\mathbb{Q}(\sqrt{d})\)
with \(3\)-class group of type \((3,3)\),
we constructed the four unramified cyclic cubic extensions \(L_i\vert K\), \(1\le i\le 4\),
and computed the IPAD \(\tau^{(1)}(K)=\lbrack 1^2;(54,21,1^3,21)\rbrack\)
with the aid of MAGMA
\cite{MAGMA}.

According to Theorem
\ref{thm:Sequences3x3},
the second \(3\)-class group \(G\) of \(K\) must be of coclass \(\mathrm{cc}(G)=2\),
and the polarized component \(54\) of the IPAD shows that \(c-k=5+4=9\) and thus
the nilpotency class \(c=\mathrm{cl}(G)\) and the defect of commutativity \(k\)
are given by either \(c=9\), \(k=0\), or \(c=10\), \(k=1\).
Further, in view of the rank-\(3\) component \(1^3\) of the IPAD,
\(G\) must be a vertex of the coclass tree \(\mathcal{T}^2(\langle 729,49\rangle)\).

When we tried to determine the \(3\)-principalization type \(\varkappa:=\varkappa_1(3,K)\),
MAGMA succeeded in calculating \(\varkappa(1)=3\) and \(\varkappa(2)=3\)
but unfortunately failed to give \(\varkappa(3)\) and \(\varkappa(4)\).
With respect to the complete IPAD, Theorem
\ref{thm:Ord243Id6}
enforces \(\varkappa(3)=1\) (item (1)) and \(\varkappa(4)=3\) (item (2)),
and therefore the partial result \(\varkappa=(3,3,\ast,\ast)\)
is completed to \(\varkappa=(3,3,1,3)\).
According to item (3) of Remark
\ref{rmk:Ord243Id6Var}
or item (3) of Corollary
\ref{cor:Ord243Id6},
\(K\) is of TKT H.4.
Our experience suggests that this TKT compels the arrangement \(c=10\), \(k=1\),
expressed by the \textit{weak leaf conjecture}
\cite[Cnj.3.1, p.423]{Ma4}.
\end{example}

\begin{example}
\label{exm:TKTE8}
For the discriminant \(d=-4\,928\,155\) of
the complex quadratic field \(K=\mathbb{Q}(\sqrt{d})\)
with \(3\)-class group of type \((3,3)\),
we constructed the four unramified cyclic cubic extensions \(L_i\vert K\), \(1\le i\le 4\),
and computed the IPAD \(\tau^{(1)}(K)=\lbrack 1^2;(21,54,(21)^2)\rbrack\)
with the aid of MAGMA
\cite{MAGMA}.

According to Theorem
\ref{thm:Sequences3x3},
the second \(3\)-class group \(G\) of \(K\) must be of coclass \(\mathrm{cc}(G)=2\),
and the polarized component \(54\) of the IPAD shows that \(c-k=5+4=9\) and thus
the nilpotency class \(c=\mathrm{cl}(G)\) and the defect of commutativity \(k\)
are given by either \(c=9\), \(k=0\), or \(c=10\), \(k=1\).
Further, due to the lack of a rank-\(3\) component \(1^3\) in the IPAD,
\(G\) must be a vertex of the coclass tree \(\mathcal{T}^2(\langle 729,54\rangle)\).

Next, we tried to determine the \(3\)-principalization type \(\varkappa:=\varkappa_1(3,K)\).
MAGMA succeeded in calculating two fixed points \(\varkappa(1)=1\) and \(\varkappa(2)=2\)
but unfortunately failed to give \(\varkappa(3)\) and \(\varkappa(4)\).
With respect to the complete IPAD, Theorem
\ref{thm:Ord243Id8}
enforces \(\varkappa(3)=3\) or \(\varkappa(4)=4\) (item (1)),
and \(\varkappa(4)=2\) or \(\varkappa(3)=2\) (item (2)),
and therefore the partial result \(\varkappa=(1,2,\ast,\ast)\)
is completed to \(\varkappa=(1,2,3,2)\) or \(\varkappa=(1,2,2,4)\).
According to item (1) of Remark
\ref{rmk:Ord243Id8Var}
or item (3) of Corollary
\ref{cor:Ord243Id8},
\(K\) is of TKT E.8,
and this TKT enforces the arrangement \(c=9\), \(k=0\),
since \(k=1\) is impossible.
\end{example}

\begin{example}
\label{exm:TKTE14}
For the discriminant \(d=-65\,433\,643\) of
the complex quadratic field \(K=\mathbb{Q}(\sqrt{d})\)
with \(3\)-class group of type \((3,3)\),
we constructed the four unramified cyclic cubic extensions \(L_i\vert K\), \(1\le i\le 4\),
and computed the IPAD \(\tau^{(1)}(K)=\lbrack 1^2;(65,1^3,(21)^2)\rbrack\)
with the aid of MAGMA
\cite{MAGMA}.

According to Theorem
\ref{thm:Sequences3x3},
the second \(3\)-class group \(G\) of \(K\) must be of coclass \(\mathrm{cc}(G)=2\),
and the polarized component \(65\) of the IPAD shows that \(c-k=6+5=11\) and thus
the nilpotency class \(c=\mathrm{cl}(G)\) and the defect of commutativity \(k\)
are given by either \(c=11\), \(k=0\), or \(c=12\), \(k=1\).
Further, in view of the rank-\(3\) component \(1^3\) of the IPAD,
\(G\) must be a vertex of the coclass tree \(\mathcal{T}^2(\langle 729,49\rangle)\).

Then we tried to determine the \(3\)-principalization type \(\varkappa:=\varkappa_1(3,K)\).
MAGMA succeeded in calculating \(\varkappa(1)=4\) and \(\varkappa(2)=1\)
but unfortunately failed to give \(\varkappa(3)\) and \(\varkappa(4)\).
With respect to the complete IPAD, Theorem
\ref{thm:Ord243Id6}
enforces \(\varkappa(3)=2\) and \(\varkappa(4)=2\) (item (2)),
whereas the claim in item (1) is confirmed,
and therefore the partial result \(\varkappa=(4,1,\ast,\ast)\)
is completed to \(\varkappa=(4,1,2,2)\).
According to item (2) of Remark
\ref{rmk:Ord243Id6Var}
or item (3) of Corollary
\ref{cor:Ord243Id6},
\(K\) is of TKT E.14,
and this TKT enforces the arrangement \(c=11\), \(k=0\),
since \(k=1\) is impossible.
\end{example}



\section{Iterated IPADs of second order}
\label{s:IPAD2ndOrd}




\subsection{\(p\)-capitulation type}
\label{ss:Capitulation}

By means of the following theorem,
the exact \(3\)-principalization type \(\varkappa\) of
real quadratic fields \(K=\mathbb{Q}(\sqrt{d})\), \(d>0\),
can be determined indirectly
with the aid of information on the structure of \(3\)-class groups
of number fields of absolute degree \(6\cdot 3=18\).

\begin{theorem}
\label{thm:IndirectCapitulationType}
(Indirect computation of the \(p\)-capitulation type)\\
Suppose that \(p=3\) and
let \(K\) be a number field with \(3\)-class group \(\mathrm{Cl}_3(K)\) of type \((3,3)\)
and \(3\)-tower group \(G\).

\begin{enumerate}
\item
If the IPAD of \(K\) is given by
\[\tau^{(1)}(K)=\lbrack 1^2;(21;(1^2)^3)\rbrack,\]
then
\[G^{\prime\prime}=1,\ G\simeq G/G^{\prime\prime},\ \mathrm{cc}(G)=1,
\text{ and }G\in\lbrace\langle 81,8\rangle,\langle 81,9\rangle,\langle 81,10\rangle\rbrace,\]
in particular, the length of the \(3\)-class tower of \(K\) is given by \(\ell_3(K)=2\).
\item
If the first layer \(\mathrm{Lyr}_1(K)\) of abelian unramified extensions of \(K\)
consists of \(L_1,\ldots,L_4\), then the iterated IPAD of second order
\[\tau^{(2)}(K)=\lbrack\tau_0(K);(\tau_0(L_i);\tau_1(L_i))_{1\le i\le 4}\rbrack,\text{ with }\tau_0(K)=1^2,\]
admits a sharp decision about the group \(G\) and the first layer of the
transfer kernel type
\[\varkappa(K)=\lbrack\varkappa_0(K);\varkappa_1(K);\varkappa_2(K);
\rbrack\text{ where trivially }\varkappa_0(K)=1,\  \varkappa_2(K)=0.\]

\begin{equation}
\label{eqn:SecondIPAD81Id10}
\begin{aligned}
\lbrack \tau_0(L_1);\tau_1(L_1)\rbrack &= \lbrack 21;(1^2,(2)^3)\rbrack, \\
\lbrack \tau_0(L_i);\tau_1(L_i)\rbrack &= \lbrack 1^2;(1^2,\mathbf{(2)^3})\rbrack,\text{ for }2\le i\le 4,
\end{aligned}
\end{equation}

\noindent
implies \(G\simeq\langle 81,10\rangle\) and thus \(\varkappa_1(K)=(1,0,0,0)\),

\begin{equation}
\label{eqn:SecondIPAD81Id8}
\begin{aligned}
\lbrack \tau_0(L_1);\tau_1(L_1)\rbrack &= \lbrack 21;(1^2,(2)^3)\rbrack, \\
\lbrack \tau_0(L_2);\tau_1(L_2)\rbrack &= \lbrack 1^2;(1^2)^4\rbrack, \\
\lbrack \tau_0(L_i);\tau_1(L_i)\rbrack &= \lbrack 1^2;(1^2,\mathbf{(2)^3})\rbrack,\text{ for }3\le i\le 4,
\end{aligned}
\end{equation}

\noindent
implies \(G\simeq\langle 81,8\rangle\) and thus \(\varkappa_1(K)=(2,0,0,0)\), and

\begin{equation}
\label{eqn:SecondIPAD81Id9}
\begin{aligned}
\lbrack \tau_0(L_1);\tau_1(L_1)\rbrack &= \lbrack 21;(1^2,(2)^3)\rbrack, \\
\lbrack \tau_0(L_i);\tau_1(L_i)\rbrack &= \lbrack 1^2;(1^2)^4\rbrack,\text{ for }2\le i\le 4,
\end{aligned}
\end{equation}

\noindent
implies \(G\simeq\langle 81,9\rangle\) and thus \(\varkappa_1(K)=(0,0,0,0)\).

\end{enumerate}

\end{theorem}



\begin{example}
\label{exm:IndirectCapitulationType}
A possible future application of Theorem
\ref{thm:IndirectCapitulationType}
could, for instance, be the separation of the capitulation types 
a.2, \(\varkappa_1(K)=(1,0,0,0)\), and
a.3, \(\varkappa_1(K)=(2,0,0,0)\),
among the \(1\,386\) 
real quadratic fields \(K=\mathbb{Q}(\sqrt{d})\), \(0<d<10^7\),
with \(3\)-class group \(\mathrm{Cl}_3(K)\) of type \((3,3)\)
and IPAD \(\tau^{(1)}(K)=\lbrack 1^2;(21;(1^2)^3)\rbrack\),
which was outside of our reach in all investigations of
\cite[Tbl.2, p.496]{Ma1},
\cite[Tbl.6.1, p.451]{Ma3}
and
\cite[Fig.3.2, p.422]{Ma4}.
The reason why we expect this enterprise to be promising
is that our experience with Magma
\cite{MAGMA}
shows that computing class groups can become slow
but remains sound and stable for huge discriminants \(d\),
whereas the calculation of capitulation kernels frequently fails.
\end{example}



\subsection{Length of the \(p\)-class tower}
\label{ss:TowerLength}

In this section, we use the iterated IPAD of second order
\(\tau^{(2)}(K)=\lbrack\tau_0(K);(\tau_0(L);\tau_1(L))_{L\in\mathrm{Lyr}_1(K)}\rbrack\)
for the indirect computation of the length \(\ell_p(K)\) of the \(p\)-class tower
of a number field \(K\), where \(p\) denotes a fixed prime.

\begin{theorem}
\label{thm:TowerLengthE6AndE14}
(Length \(\ell_3(K)\) of the \(3\)-class tower for \(G/G^{\prime\prime}\in\mathcal{T}^2(\langle 243,6\rangle)\))\\
Suppose that \(p=3\) and
let \(K\) be a number field with \(3\)-class group \(\mathrm{Cl}_3(K)\) of type \((3,3)\)
and \(3\)-tower group \(G\).

\begin{enumerate}
\item
If the IPAD of \(K\) is given by
\[\tau^{(1)}(K)=\lbrack 1^2;(32;1^3,(21)^2)\rbrack,\]
and the first layer TKT \(\varkappa_1(K)\)
neither contains a total principalization nor a \(2\)-cycle,
then there are two possibilities \(\ell_3(K)\in\lbrace 2,3\rbrace\)
for the length of the \(3\)-class tower of \(K\).
\item
If the first layer \(\mathrm{Lyr}_1(K)\) of abelian unramified extensions of \(K\)
consists of \(L_1,\ldots,L_4\), then the iterated IPAD of second order
\[\tau^{(2)}(K)=\lbrack\tau_0(K);(\tau_0(L_i);\tau_1(L_i))_{1\le i\le 4}\rbrack,\text{ with }\tau_0(K)=1^2,\]
admits a sharp decision about the length \(\ell_3(K)\):

\begin{equation}
\label{eqn:SecondIPADE6E14Length2}
\begin{aligned}
\lbrack \tau_0(L_1);\tau_1(L_1)\rbrack &= \lbrack 32;(2^21,(31^2)^3)\rbrack, \\
\lbrack \tau_0(L_2);\tau_1(L_2)\rbrack &= \lbrack 1^3;(2^21,\mathbf{(1^3)^3},(1^2)^9)\rbrack, \\
\lbrack \tau_0(L_i);\tau_1(L_i)\rbrack &= \lbrack 21;(2^21,\mathbf{(21)^3})\rbrack,\text{ for }3\le i\le 4,
\end{aligned}
\end{equation}

\noindent
if and only if \(\ell_3(K)=2\), and

\begin{equation}
\label{eqn:SecondIPADE6E14Length3}
\begin{aligned}
\lbrack \tau_0(L_1);\tau_1(L_1)\rbrack &= \lbrack 32;(2^21,(31^2)^3)\rbrack, \\
\lbrack \tau_0(L_2);\tau_1(L_2)\rbrack &= \lbrack 1^3;(2^21,\mathbf{(21^2)^3},(1^2)^9)\rbrack, \\
\lbrack \tau_0(L_i);\tau_1(L_i)\rbrack &= \lbrack 21;(2^21,\mathbf{(31)^3})\rbrack,\text{ for }3\le i\le 4,
\end{aligned}
\end{equation}

\noindent
if and only if \(\ell_3(K)=3\).

\end{enumerate}
\end{theorem}

\begin{proof}
According to Theorem
\ref{thm:Sequences3x3},
an IPAD of the form \(\tau^{(1)}(K)=\lbrack 1^2;(32;1^3,(21)^2)\rbrack\)
indicates that the metabelianization of the group \(G\) belongs to
the coclass tree \(\mathcal{T}^2(\langle 243,6\rangle)\)
\cite[Fig.3.6, p.442]{Ma4}
and has nilpotency class \(3+2=5\), due to the polarization.

According to \S\
\ref{ss:Ord243Id6},
the lack of a total principalization excludes the TKT c.18
and the absence of a \(2\)-cycle discourages the TKT H.4,
whence the group \(G\) must be of TKT E.6 or E.14.

By means of the techniques described in
\cite{BuMa},
a search in the complete descendant tree \(\mathcal{T}(\langle 243,6\rangle)\),
not restricted to groups of coclass \(2\),
yields exactly six candidates for the group \(G\):
three metabelian groups \(\langle 2187,i\rangle\) with \(i\in\lbrace 288,289,290\rbrace\),
and three groups of derived length \(3\) and order \(3^8\) with generalized identifiers
\(\langle 729,49\rangle-\#2;i\), \(i\in\lbrace 4,5,6\rbrace\).
There cannot exist adequate groups of bigger orders.
The former three groups are charcterized by Equations
(\ref{eqn:SecondIPADE6E14Length2})
the latter three groups
(see
\cite[\S\ 20.2, Fig.8]{Ma6})
by Equations
(\ref{eqn:SecondIPADE6E14Length3}).

Finally, we have \(\ell_3(K)=\mathrm{dl}(G)\).
\end{proof}



\begin{theorem}
\label{thm:TowerLengthE8AndE9}
(Length \(\ell_p(K)\) of the \(3\)-class tower for \(G/G^{\prime\prime}\in\mathcal{T}^2(\langle 243,8\rangle)\))\\
Suppose that \(p=3\) and
let \(K\) be a number field with \(3\)-class group \(\mathrm{Cl}_3(K)\) of type \((3,3)\)
and \(3\)-tower group \(G\).

\begin{enumerate}

\item
If the IPAD of \(K\) is given by
\[\tau^{(1)}(K)=\lbrack 1^2;(32;(21)^3)\rbrack,\]
and the first layer TKT \(\varkappa_1(K)\)
neither contains a total principalization nor a \(2\)-cycle,
then there are two possibilities \(\ell_3(K)\in\lbrace 2,3\rbrace\)
for the length of the \(3\)-class tower of \(K\).

\item
If the first layer \(\mathrm{Lyr}_1(K)\) of abelian unramified extensions of \(K\)
consists of \(L_1,\ldots,L_4\), then the iterated IPAD of second order
\[\tau^{(2)}(K)=\lbrack\tau_0(K);(\tau_0(L_i);\tau_1(L_i))_{1\le i\le 4}\rbrack,\text{ with }\tau_0(K)=1^2,\]
admits a sharp decision about the length \(\ell_3(K)\):

\begin{equation}
\label{eqn:SecondIPADE8E9Length2}
\begin{aligned}
\lbrack \tau_0(L_1);\tau_1(L_1)\rbrack &= \lbrack 32;(2^21,(31^2)^3)\rbrack, \\
\lbrack \tau_0(L_i);\tau_1(L_i)\rbrack &= \lbrack 21;(2^21,\mathbf{(21)^3})\rbrack,\text{ for }2\le i\le 4,
\end{aligned}
\end{equation}

\noindent
if and only if \(\ell_3(K)=2\), and

\begin{equation}
\label{eqn:SecondIPADE8E9Length3}
\begin{aligned}
\lbrack \tau_0(L_1);\tau_1(L_1)\rbrack &= \lbrack 32;(2^21,(31^2)^3)\rbrack, \\
\lbrack \tau_0(L_i);\tau_1(L_i)\rbrack &= \lbrack 21;(2^21,\mathbf{(31)^3})\rbrack,\text{ for }2\le i\le 4,
\end{aligned}
\end{equation}

\noindent
if and only if \(\ell_3(K)=3\).

\end{enumerate}
\end{theorem}

\begin{proof}
According to Theorem
\ref{thm:Sequences3x3},
an IPAD of the form \(\tau^{(1)}(K)=\lbrack 1^2;(32;(21)^3)\rbrack\)
indicates that the metabelianization of the group \(G\) belongs to
the coclass tree \(\mathcal{T}^2(\langle 243,8\rangle)\)
\cite[Fig.3.7, p.443]{Ma4}
and has nilpotency class \(3+2=5\), due to the polarization.

According to \S\
\ref{ss:Ord243Id8},
the lack of a total principalization excludes the TKT c.21
and the absence of a \(2\)-cycle discourages the TKT G.16,
whence the group \(G\) must be of TKT E.8 or E.9.

As we have shown in detail in
\cite{BuMa},
a search in the complete descendant tree \(\mathcal{T}(\langle 243,8\rangle)\),
not restricted to groups of coclass \(2\),
yields exactly six candidates for the group \(G\):
three metabelian groups \(\langle 2187,i\rangle\) with \(i\in\lbrace 302,304,306\rbrace\),
and three groups of derived length \(3\) and order \(3^8\) with generalized identifiers
\(\langle 729,54\rangle-\#2;i\), \(i\in\lbrace 2,4,6\rbrace\).
There cannot exist adequate groups of bigger orders.
The former three groups are characterized by Equations
(\ref{eqn:SecondIPADE8E9Length2})
the latter three groups
(see
\cite[\S\ 20.2, Fig.9]{Ma6})
by Equations
(\ref{eqn:SecondIPADE8E9Length3}).

Eventually, the \(3\)-tower length of \(K\), \(\ell_3(K)=\mathrm{dl}(G)\),
coincides with the derived length of \(G\).
\end{proof}



\begin{example}
\label{exm:TowerLengthE8AndE9}
In June 2006, we discovered the smallest discriminant \(d=342\,664\)
of a real quadratic field \(K=\mathbb{Q}(\sqrt{d})\) with \(3\)-class group of type \((3,3)\)
whose \(3\)-tower group \(G\) possesses the transfer kernel type E.9,
\(\varkappa=(2,3,3,4)\).

The complex quadratic analogue \(k=\mathbb{Q}(\sqrt{-9\,748})\) was known since 1934
by the famous paper of Scholz and Taussky
\cite{SoTa}.
However, it required almost 80 years until M.R. Bush and ourselves
\cite{BuMa}
succeeded in providing the first faultless proof that \(k\) has
a \(3\)-class tower of exact length \(\ell_3(k)=3\) with \(3\)-tower group \(G\)
one of the two Schur \(\sigma\)-groups
\(\langle 729,54\rangle-\#2;i\), \(i\in\lbrace 2,6\rbrace\), of order \(3^8\).

For \(K=\mathbb{Q}(\sqrt{342\,664})\), the methods in
\cite{BuMa}
do not admit a final decision about the length \(\ell_3(K)\).
They only yield four possible \(3\)-tower groups of \(K\), namely
either the two unbalanced groups
\(\langle 2187,i\rangle\) with \(i\in\lbrace 302,306\rbrace\)
and relation rank \(r=3\) bigger than the generator rank \(d=2\)
or the two Schur \(\sigma\)-groups
\(\langle 729,54\rangle-\#2;i\) with \(i\in\lbrace 2,6\rbrace\)
and \(r=2\) equal to \(d=2\).

In October 2014, we succeeded in proving that
three of the unramified cyclic cubic extensions \(L_i\vert K\)
reveal the critical IPAD component
\(\tau_1(L_i)=(2^21,(31)^3)\) in Equation
(\ref{eqn:SecondIPADE8E9Length3})
of Theorem
\ref{thm:TowerLengthE8AndE9},
item (2), whence \(\ell_3(K)=3\).
This was done by computing \(3\)-class groups of number fields
of absolute degree \(6\cdot 3=18\) with the aid of MAGMA
\cite{MAGMA}.
\end{example}



L. Bartholdi and M.R. Bush
\cite{BaBu}
have shown that the unbalanced metabelian \(3\)-group \(G=\langle 729,45\rangle\)
possesses an infinite balanced cover \(\mathrm{cov}_\ast(G)\)
\cite[Dfn.21.2]{Ma6},
which implies that the length \(\ell_3(K)\) of the \(3\)-class tower
of a complex quadratic field \(K\) with IPAD \(\tau(K)=\lbrack 1^2;((1^3)^3,21)\rbrack\)
can take any value bigger than \(2\) or even \(\infty\).
The group theoretic reason for this remarkable extravagance is that \(G\)
is not coclass-settled and gives rise to a descendant tree \(\mathcal{T}(G)\)
which contains infinitely many periodic bifurcations
\cite[\S\ 21]{Ma6}.

As a final coronation of this section,
we show that our new IPAD strategies are powerful enough
to enable the determination of the length \(\ell_3(K)\)
with the aid of information on the structure of \(3\)-class groups
of number fields of absolute degree \(6\cdot 9=54\).

For this purpose, we extend the concept of iterated IPADs of second order
\[\tau^{(1)}(K)=\lbrack\tau_0(K);(\tau^{(1)}(L))_{L\in\mathrm{Lyr}_1(K)}\rbrack
=\lbrack\tau_0(K);(\tau_0(L);\tau_1(L))_{L\in\mathrm{Lyr}_1(K)}\rbrack\]
once more by adding the second layers \(\tau_2(L)\) of all IPADs \(\tau^{(1)}(L)\)
of unramified degree-\(p\) extensions \(L\vert K\).
The resulting \textit{iterated multi-layered IPAD of second order} is indicated by an asterisk
\[\tau_\ast^{(1)}(K)=\lbrack\tau_0(K);(\tau_0(L);\tau_1(L);\tau_2(L))_{L\in\mathrm{Lyr}_1(K)}\rbrack.\]



\begin{theorem}
\label{thm:TowerLengthH4}
(Length \(\ell_p(K)\) of the \(3\)-class tower for \(G/G^{\prime\prime}\in\mathcal{T}(\langle 243,4\rangle)\))\\
Suppose that \(p=3\) and
let \(K\) be a number field with \(3\)-class group \(\mathrm{Cl}_3(K)\) of type \((3,3)\)
and \(3\)-tower group \(G\).

\begin{enumerate}

\item
If the IPAD of \(K\) is given by
\[\tau^{(1)}(K)=\lbrack 1^2;((1^3)^3,21)\rbrack,\]
then the first layer TKT is \(\varkappa_1(K)=(4,1,1,1)\)
and there exist infinitely many possibilities \(\ell_3(K)\ge 2\)
for the length of the \(3\)-class tower of \(K\).

\item
If the first layer \(\mathrm{Lyr}_1(K)\) of abelian unramified extensions of \(K\)
consists of \(L_1,\ldots,L_4\), then the iterated multi-layered IPAD of second order
\[\tau_{\ast}^{(2)}(K)=\lbrack\tau_0(K);(\tau_0(L_i);\tau_1(L_i);\tau_2(L_i))_{1\le i\le 4}\rbrack,
\text{ with }\tau_0(K)=1^2,\]
admits certain partial decisions about the length \(\ell_3(K)\):

\begin{equation}
\label{eqn:SecondIPADH4Ord243Id4}
\begin{aligned}
\lbrack \tau_0(L_1);\tau_1(L_1);\tau_2(L_1)\rbrack &= \lbrack 1^3;((1^3)^4,(1^2)^9);(1^2)^{13})\rbrack, \\
\lbrack \tau_0(L_i);\tau_1(L_i);\tau_2(L_i)\rbrack &= \lbrack 1^3;(1^3,(21)^3,(1^2)^9);((1^2)^4,(2)^9)\rbrack,\text{ for }2\le i\le 3, \\
\lbrack \tau_0(L_4);\tau_1(L_4);\tau_2(L_4)\rbrack &= \lbrack 21;(1^3,(21)^3);(1^2)^4\rbrack
\end{aligned}
\end{equation}

\noindent
implies \(G\simeq\langle 243,4\rangle\) and \(\ell_3(K)=2\).

\begin{equation}
\label{eqn:SecondIPADH4Ord729Id45}
\begin{aligned}
\lbrack \tau_0(L_1);\tau_1(L_1);\tau_2(L_1)\rbrack &= \lbrack 1^3;(\mathbf{21^2},(1^3)^3,(1^2)^9);(1^3,(21)^3,(1^2)^9)\rbrack, \\
\lbrack \tau_0(L_i);\tau_1(L_i);\tau_2(L_i)\rbrack &= \lbrack 1^3;\mathbf{(21^2,(21)^{12})};(1^3,(21)^{12})\rbrack,\text{ for }2\le i\le 3, \\
\lbrack \tau_0(L_4);\tau_1(L_4);\tau_2(L_4)\rbrack &= \lbrack 21;(\mathbf{21^2},(21)^3);(1^3,(21)^3)\rbrack
\end{aligned}
\end{equation}

\noindent
implies \(G\simeq\langle 729,45\rangle\) and \(\ell_3(K)=2\).

\begin{equation}
\label{eqn:SecondIPADH4Ord2187Id273}
\begin{aligned}
\lbrack \tau_0(L_1);\tau_1(L_1);\tau_2(L_1)\rbrack &= \lbrack 1^3;(21^2,(1^3)^3,(1^2)^9);(21^2,(21)^3,(1^2)^9)\rbrack, \\
\lbrack \tau_0(L_2);\tau_1(L_2);\tau_2(L_2)\rbrack &= \lbrack 1^3;(21^2,(21)^{12});(21^2,(21)^{12})\rbrack, \\
\lbrack \tau_0(L_3);\tau_1(L_3);\tau_2(L_3)\rbrack &= \lbrack 1^3;\mathbf{((21^2)^4,(2^2)^9)};(21^2)^{13}\rbrack, \\
\lbrack \tau_0(L_4);\tau_1(L_4);\tau_2(L_4)\rbrack &= \lbrack 21;(21^2,(21)^3);(21^2,(21)^3)\rbrack
\end{aligned}
\end{equation}

\noindent
implies \(G\simeq\langle 2187,273\rangle\) and \(\ell_3(K)=3\).

\begin{equation}
\label{eqn:SecondIPADH4Ord729Id45Stp2No2}
\begin{aligned}
\lbrack \tau_0(L_1);\tau_1(L_1);\tau_2(L_1)\rbrack &= \lbrack 1^3;((21^2)^4,(1^2)^9);(2^21,(1^3)^3,\mathbf{(2^2)^3},(21)^6)\rbrack, \\
\lbrack \tau_0(L_i);\tau_1(L_i);\tau_2(L_i)\rbrack &= \lbrack 1^3;((21^2)^4,(2^2)^9);(2^21,\mathbf{(21^2)^{12}})\rbrack,\text{ for }2\le i\le 3, \\
\lbrack \tau_0(L_4);\tau_1(L_4);\tau_2(L_4)\rbrack &= \lbrack 21;(21^2,(31)^3);(2^21,\mathbf{(2^2)^3})\rbrack
\end{aligned}
\end{equation}

\noindent
implies \(G\simeq\langle 729,45\rangle-\#2;2\) of order \(3^8\) and \(\ell_3(K)=3\).

\begin{equation}
\label{eqn:SecondIPADH4Ord729Id45HigherDesc}
\begin{aligned}
\lbrack \tau_0(L_1);\tau_1(L_1);\tau_2(L_1)\rbrack &= \lbrack 1^3;((21^2)^4,(1^2)^9);(2^21,(1^3)^3,\mathbf{(32)^3},(21)^6)\rbrack, \\
\lbrack \tau_0(L_i);\tau_1(L_i);\tau_2(L_i)\rbrack &= \lbrack 1^3;((21^2)^4,(2^2)^9);(\mathbf{(2^21)^4,(31^2)^9})\rbrack,\text{ for }2\le i\le 3, \\
\lbrack \tau_0(L_4);\tau_1(L_4);\tau_2(L_4)\rbrack &= \lbrack 21;(21^2,(31)^3);(2^21,\mathbf{(32)^3})\rbrack
\end{aligned}
\end{equation}

\noindent
implies either \\
\(G\simeq\langle 729,45\rangle(-\#2;1-\#1;2)^j-\#2;2\), \(1\le j\le 2\), of order \(3^{8+3j}\) and \(\ell_3(K)=3\) \\
or \(G\simeq\langle 729,45\rangle(-\#2;1-\#1;2)^3-\#2;2\) of order \(3^{17}\) and \(\ell_3(K)=4\).

\end{enumerate}
\end{theorem}



\begin{example}
\label{exm:TowerLengthH4}
In December 2009, we discovered the smallest discriminant \(d=957\,013\)
of a real quadratic field \(K=\mathbb{Q}(\sqrt{d})\) with \(3\)-class group of type \((3,3)\)
whose \(3\)-tower group \(G\) possesses the transfer kernel type H.4,
\(\varkappa=(4,1,1,1)\).
The complex quadratic analogue \(K=\mathbb{Q}(\sqrt{-3\,896})\) was known since 1982
by the paper of Heider and Schmithals
\cite{HeSm}.
Both fields share the same IPAD \(\tau^{(1)}(K)=\lbrack 1^2;((1^3)^3,21)\rbrack\).

In February 2015, we succeeded in proving that
the unramified cyclic cubic extensions \(L_i\vert K\),
for \(d=-3\,896\), resp. \(d=957\,013\),
reveal the critical (first and) second layer IPAD components\\
\(\tau_2(L_1)=(2^21,(1^3)^3,\mathbf{(2^2)^3},(21)^6)\),\\
\(\tau_2(L_i)=(2^21,\mathbf{(21^2)^{12}}),\text{ for }2\le i\le 3\), and\\
\(\tau_2(L_4)=(2^21,\mathbf{(2^2)^3})\)\\
in Equation
(\ref{eqn:SecondIPADH4Ord729Id45Stp2No2}),
resp.\\
\(\lbrack\tau_1(L_1);\tau_2(L_1)\rbrack=\lbrack(21^2,(1^3)^3,(1^2)^9);(21^2,(21)^3,(1^2)^9)\rbrack\),\\
\(\lbrack\tau_1(L_2);\tau_2(L_2)\rbrack=\lbrack(21^2,(21)^{12});(21^2,(21)^{12})\rbrack\),\\
\(\lbrack\tau_1(L_3);\tau_2(L_3)\rbrack=\lbrack\mathbf{((21^2)^4,(2^2)^9)};(21^2)^{13}\rbrack\),\\
\(\lbrack\tau_1(L_4);\tau_2(L_4)\rbrack=\lbrack(21^2,(21)^3);(21^2,(21)^3)\rbrack\)\\
in Equation
(\ref{eqn:SecondIPADH4Ord2187Id273}),
of Theorem
\ref{thm:TowerLengthH4},
item (2), whence \(\ell_3(K)=3\), for both fields.

However, the \(3\)-class tower groups are different:\\
\(K=\mathbb{Q}(\sqrt{-3\,896})\) has the Schur \(\sigma\)-group \(G\simeq\langle 729,45\rangle-\#2;2\) of order \(3^8\), and\\
\(K=\mathbb{Q}(\sqrt{957\,013})\) has the unbalanced group \(G\simeq\langle 2187,273\rangle\).

This was done by computing \(3\)-class groups of number fields
of absolute degree \(6\cdot 9=54\) with the aid of MAGMA
\cite{MAGMA}.
\end{example}



\section{Complex quadratic fields of \(3\)-rank three}
\label{s:CmpQdr3Rk3}

In this concluding section we present another impressive application of IPADs.

Due to Koch and Venkov
\cite{KoVe},
it is known that a complex quadratic field \(K\) with \(3\)-class rank \(r_3(K)\ge 3\)
has an infinite \(3\)-class field tower
\(K<\mathrm{F}_3^1(K)<\mathrm{F}_3^2(K)<\ldots<\mathrm{F}_3^{\infty}(K)\)
of length \(\ell_3(K)=\infty\).
In the time between 1973 and 1978,
Diaz y Diaz
\cite{DD1,DD2}
and Buell
\cite{Bl}
have determined the smallest absolute discriminants \(\lvert d\rvert\) of such fields.
Recently, we have launched a computational project which aims at
verifying these classical results and adding sophisticated arithmetical details.
Below the bound \(10^7\) there exist \(25\) discriminants \(d\) of this kind,
and \(14\) of the corresponding fields \(K\) have a \(3\)-class group \(\mathrm{Cl}_3(K)\)
of elementary abelian type \((3,3,3)\).
For each of these \(14\) fields,
we determine the type of \(3\)-principalization \(\varkappa:=\varkappa_1(3,K)\) in the
thirteen unramified cyclic cubic extensions \(L_1,\ldots,L_{13}\) of \(K\),
and the structure of the \(3\)-class groups  \(\mathrm{Cl}_3(L_i)\) of these extensions,
i.e., the IPAD of \(K\).
We characterize the metabelian Galois group \(G=\mathrm{G}_3^2(K)=\mathrm{Gal}(\mathrm{F}_3^2(K)\vert K)\)
of the second Hilbert \(3\)-class field \(\mathrm{F}_3^2(K)\)
by means of kernels and targets of its Artin transfer homomorphisms
\cite{Ar2}
to maximal subgroups.
We provide evidence of a wealth of structure in the set of
infinite topological \(3\)-class field tower groups
\(\mathrm{G}_3^{\infty}(K)=\mathrm{Gal}(\mathrm{F}_3^{\infty}(K)\vert K)\)
by showing that the \(14\) groups \(G\) are pairwise non-isomorphic.



\renewcommand{\arraystretch}{1.1}

\begin{table}[ht]
\caption{Data collection for \(\mathrm{Cl}_3(K)\simeq (3,3,3)\) and \(-10^7<d\)}
\label{tbl:DataCollectionRank3}
\begin{center}
\begin{tabular}{|r|r|r|r|}
\hline
   No.  & discriminant \(d\) & \(\mathrm{Cl}_3(K)\) & \(\mathrm{Cl}(K)\) \\
\hline
  \(1\) &   \(-3\,321\,607\) &          \((9,3,3)\) &       \((63,3,3)\) \\
  \(2\) &   \(-3\,640\,387\) &          \((9,3,3)\) &       \((18,3,3)\) \\
  \(3\) &   \(-4\,019\,207\) &          \((9,3,3)\) &      \((207,3,3)\) \\
  \(4\) &   \(-4\,447\,704\) &          \((3,3,3)\) &       \((24,6,6)\) \\
  \(5\) &   \(-4\,472\,360\) &          \((3,3,3)\) &       \((30,6,6)\) \\
  \(6\) &   \(-4\,818\,916\) &          \((3,3,3)\) &       \((48,3,3)\) \\
  \(7\) &   \(-4\,897\,363\) &          \((3,3,3)\) &       \((33,3,3)\) \\
  \(8\) &   \(-5\,048\,347\) &          \((9,3,3)\) &       \((18,6,3)\) \\
  \(9\) &   \(-5\,067\,967\) &          \((3,3,3)\) &       \((69,3,3)\) \\
 \(10\) &   \(-5\,153\,431\) &         \((27,3,3)\) &      \((216,3,3)\) \\
 \(11\) &   \(-5\,288\,968\) &          \((9,3,3)\) &       \((72,3,3)\) \\
 \(12\) &   \(-5\,769\,988\) &          \((3,3,3)\) &       \((12,6,6)\) \\
 \(13\) &   \(-6\,562\,327\) &          \((9,3,3)\) &      \((126,3,3)\) \\
 \(14\) &   \(-7\,016\,747\) &          \((9,3,3)\) &       \((99,3,3)\) \\
 \(15\) &   \(-7\,060\,148\) &          \((3,3,3)\) &       \((60,6,3)\) \\
 \(16\) &   \(-7\,503\,391\) &          \((9,3,3)\) &       \((90,6,3)\) \\
 \(17\) &   \(-7\,546\,164\) &          \((9,3,3)\) &     \((18,6,6,2)\) \\
 \(18\) &   \(-8\,124\,503\) &          \((9,3,3)\) &      \((261,3,3)\) \\
 \(19\) &   \(-8\,180\,671\) &          \((3,3,3)\) &      \((159,3,3)\) \\
 \(20\) &   \(-8\,721\,735\) &          \((3,3,3)\) &       \((60,6,6)\) \\
 \(21\) &   \(-8\,819\,519\) &          \((3,3,3)\) &      \((276,3,3)\) \\
 \(22\) &   \(-8\,992\,363\) &          \((3,3,3)\) &       \((48,3,3)\) \\
 \(23\) &   \(-9\,379\,703\) &          \((3,3,3)\) &      \((210,3,3)\) \\
 \(24\) &   \(-9\,487\,991\) &          \((3,3,3)\) &      \((381,3,3)\) \\
 \(25\) &   \(-9\,778\,603\) &          \((3,3,3)\) &       \((48,3,3)\) \\
\hline
\end{tabular}
\end{center}
\end{table}



We summarize our results and their obvious conclusion in the following theorem.

\begin{theorem}
\label{thm:14NonIsomGrps}

There exist exactly \(14\) complex quadratic number fields \(K=\mathbb{Q}(\sqrt{d})\)
with \(3\)-class groups \(\mathrm{Cl}_3(K)\) of type \((3,3,3)\)
and discriminants in the range \(-10^7<d<0\).
They have pairwise non-isomorphic

\begin{enumerate}
\item
second and higher \(3\)-class groups
\(\mathrm{Gal}(\mathrm{F}_3^n(K)\vert K)\), \(n\ge 2\),
\item
infinite topological \(3\)-class field tower groups \(\mathrm{Gal}(\mathrm{F}_3^{\infty}(K)\vert K)\).
\end{enumerate}

\end{theorem}

Before we come to the proof of Theorem
\ref{thm:14NonIsomGrps}
in \S\
\ref{ss:Proof14NonIsomGrps},
we collect basic numerical data concerning fields with \(r_3(K)=3\) in \S\
\ref{ss:DataCollectionRank3},
and we completely determine sophisticated arithmetical invariants in \S\
\ref{ss:PatternRecognitionRank3}
for all fields with \(\mathrm{Cl}_3(K)\) of type \((3,3,3)\).
The first attempt to do so
for the smallest absolute discriminant \(\lvert d\rvert=3\,321\,607\) with \(r_3(K)=3\)
is due to Heider and Schmithals in
\cite[\S\ 4, Tbl.2, p.18]{HeSm},
but it resulted in partial success only.



\subsection{Discriminants \(-10^7<d<0\) of fields \(K=\mathbb{Q}(\sqrt{d})\) with rank \(r_3(K)=3\)}
\label{ss:DataCollectionRank3}

Since one of our aims is
to investigate tendencies for the coclass of second and higher \(p\)-class groups
\(\mathrm{G}_p^n(K)=\mathrm{Gal}(\mathrm{F}_p^n(K)\vert K)\), \(n\ge 2\),
\cite{Ma1,Ma4}
of a series of algebraic number fields \(K\) with infinite \(p\)-class field tower,
for an odd prime \(p\ge 3\),
the most obvious choice which suggests itself is to take the smallest possible prime \(p=3\)
and to select complex quadratic fields \(K=\mathbb{Q}(\sqrt{d})\), \(d<0\),
having the simplest possible \(3\)-class group \(\mathrm{Cl}_3(K)\) of rank \(3\),
that is, of elementary abelian type \((3,3,3)\).

The reason is that Koch and Venkov
\cite{KoVe}
have improved the lower bound of Golod, Shafarevich
\cite{Sh,GoSh}
and Vinberg
\cite{Vb}
for the \(p\)-class rank \(r_p(K)\),
which ensures an infinite \(p\)-class tower of a complex quadratic field \(K\),
from \(4\) to \(3\).

However, quadratic fields with \(3\)-rank \(r_3(K)=3\) are sparse.
Diaz y Diaz and Buell
\cite{DD1,Sk,Bl,DD2}
have determined the minimal absolute discriminant of such fields to be
\(\lvert d\rvert=3\,321\,607\).

To provide an independent verification, we used the computational algebra system Magma
\cite{BCP,BCFS,MAGMA}
for compiling a list of all quadratic fundamental discriminants \(-10^7<d<0\)
of fields \(K=\mathbb{Q}(\sqrt{d})\) with \(3\)-class rank \(r_3(K)=3\).
In \(16\) hours of CPU time we obtained the \(25\) desired discriminants
and the abelian type invariants (here written in \(3\)-power form)
of the corresponding \(3\)-class groups \(\mathrm{Cl}_3(K)\),
and also of the complete class groups \(\mathrm{Cl}(K)\),
as given in Table
\ref{tbl:DataCollectionRank3}.
There appeared only one discriminant \(d=-7\,503\,391\) (No. \(16\)) which is not contained in
\cite[Appendix 1, p.68]{DD2}
already.

There are \(14\) discriminants, starting with \(d=-4\,447\,704\),
such that \(\mathrm{Cl}_3(K)\) is elementary abelian of type \((3,3,3)\),
and \(10\) discriminants, starting with \(-3\,321\,607\),
such that \(\mathrm{Cl}_3(K)\) is of non-elementary type \((9,3,3)\).
For the single discriminant \(d=-5\,153\,431\), we have a \(3\)-class group of type \((27,3,3)\).
We have published this information in the Online Encyclopedia of Integer Sequences (OEIS)
\cite{OEIS},
sequences A244574 and A244575.



\renewcommand{\arraystretch}{1.1}

\begin{table}[ht]
\caption{Pattern recognition via ordered IPADs}
\label{tbl:PatternRecognitionRank3}
\begin{center}
\begin{tabular}{|r|r|r|r|r|r|r|r|r|r|r|r|r|r|}
\hline
              No. & \multicolumn{13}{|c|}{discriminant} \\
\hline
            \(i\) & \(1\) & \(2\) & \(3\) & \(4\) & \(5\) & \(6\) & \(7\) & \(8\) & \(9\) & \(10\) & \(11\) & \(12\) & \(13\) \\
\hline
            \(1\) & \multicolumn{13}{|c|}{\(d=-4\,447\,704\)} \\
\hline
    \(\varkappa\) & \(8\) & \(1\) & \(8\) & \(8\) & \(10\) & \(8\) & \(6\) & \(13\) & \(8\) & \(2\) & \(10\) & \(8\) & \(9\) \\
 \(o(\varkappa)\) & \(1\) & \(1\) & \(0\) & \(0\) & \(0\) & \(1\) & \(0\) & \(6\) & \(1\) & \(2\) & \(0\) & \(0\) & \(1\) \\
         \(\tau\) & \(2^21^2\) & \(21^4\) & \(2^21^2\) & \(32^21\) & \(2^21^2\) & \(21^4\) & \(2^21^2\) & \(2^21^2\) & \(21^4\) & \(21^4\) & \(21^4\) & \(2^21^2\) & \(2^21^2\) \\
       \(\tau^0\) & \(1^2\) & \(1^2\) & \(1^2\) & \(21\) & \(1^2\) & \(1^2\) & \(1^2\) & \(1^2\) & \(1^2\) & \(1^2\) & \(1^2\) & \(1^2\) & \(1^2\) \\
\hline
            \(2\) & \multicolumn{13}{|c|}{\(d=-4\,472\,360\)} \\
\hline
    \(\varkappa\) & \(1\) & \(12\) & \(6\) & \(13\) & \(6\) & \(10\) & \(4\) & \(13\) & \(10\) & \(10\) & \(1\) & \(8\) & \(4\) \\
 \(o(\varkappa)\) & \(2\) & \(0\) & \(0\) & \(2\) & \(0\) & \(2\) & \(0\) & \(1\) & \(0\) & \(3\) & \(0\) & \(1\) & \(2\) \\
         \(\tau\) & \(2^21^2\) & \(2^21^2\) & \(2^21^2\) & \(21^4\) & \(2^21^2\) & \(21^4\) & \(21^4\) & \(2^21^2\) & \(2^21^2\) & \(2^21^2\) & \(32^21\) & \(2^21^2\) & \(21^4\) \\
       \(\tau^0\) & \(1^2\) & \(1^2\) & \(1^2\) & \(1^2\) & \(1^2\) & \(1^2\) & \(1^2\) & \(1^2\) & \(1^2\) & \(1^2\) & \(21\) & \(1^2\) & \(1^2\) \\
\hline
            \(3\) & \multicolumn{13}{|c|}{\(d=-4\,818\,916\)} \\
\hline
    \(\varkappa\) & \(6\) & \(9\) & \(13\) & \(1\) & \(5\) & \(6\) & \(9\) & \(4\) & \(11\) & \(7\) & \(1\) & \(3\) & \(4\) \\
 \(o(\varkappa)\) & \(2\) & \(0\) & \(1\) & \(2\) & \(1\) & \(2\) & \(1\) & \(0\) & \(2\) & \(0\) & \(1\) & \(0\) & \(1\) \\
         \(\tau\) & \(2^21^2\) & \(2^21^2\) & \(2^21^2\) & \(21^4\) & \(21^4\) & \(2^21^2\) & \(2^21^2\) & \(2^21^2\) & \(431^3\) & \(2^21^2\) & \(21^4\) & \(32^21\) & \(2^21^2\) \\
       \(\tau^0\) & \(1^2\) & \(1^2\) & \(1^2\) & \(1^2\) & \(1^2\) & \(1^2\) & \(1^2\) & \(1^2\) & \(31\) & \(1^2\) & \(1^2\) & \(21\) & \(1^2\) \\
\hline
            \(4\) & \multicolumn{13}{|c|}{\(d=-4\,897\,363\)} \\
\hline
    \(\varkappa\) & \(3\) & \(8\) & \(11\) & \(2\) & \(6\) & \(6\) & \(12\) & \(7\) & \(2\) & \(2\) & \(9\) & \(13\) & \(6\) \\
 \(o(\varkappa)\) & \(0\) & \(3\) & \(1\) & \(0\) & \(0\) & \(3\) & \(1\) & \(1\) & \(1\) & \(0\) & \(1\) & \(1\) & \(1\) \\
         \(\tau\) & \(2^21^2\) & \(321^3\) & \(431^3\) & \(2^21^2\) & \(21^4\) & \(2^21^2\) & \(32^21\) & \(21^4\) & \(2^21^2\) & \(2^21^2\) & \(2^21^2\) & \(2^21^2\) & \(2^21^2\) \\
       \(\tau^0\) & \(1^2\) & \(21\) & \(31\) & \(1^2\) & \(1^2\) & \(1^2\) & \(21\) & \(1^2\) & \(1^2\) & \(1^2\) & \(1^2\) & \(1^2\) & \(1^2\) \\
\hline
            \(5\) & \multicolumn{13}{|c|}{\(d=-5\,067\,967\)} \\
\hline
    \(\varkappa\) & \(8\) & \(6\) & \(9\) & \(2\) & \(3\) & \(7\) & \(12\) & \(7\) & \(1\) & \(4\) & \(3\) & \(9\) & \(4\) \\
 \(o(\varkappa)\) & \(1\) & \(1\) & \(2\) & \(2\) & \(0\) & \(1\) & \(2\) & \(1\) & \(2\) & \(0\) & \(0\) & \(1\) & \(0\) \\
         \(\tau\) & \(21^4\) & \(21^4\) & \(2^21^2\) & \(21^4\) & \(2^21^2\) & \(2^21^2\) & \(2^21^2\) & \(32^21\) & \(2^21^2\) & \(21^4\) & \(21^4\) & \(2^21^2\) & \(2^21^2\) \\
       \(\tau^0\) & \(1^2\) & \(1^2\) & \(1^2\) & \(1^2\) & \(1^2\) & \(1^2\) & \(1^2\) & \(21\) & \(1^2\) & \(1^2\) & \(1^2\) & \(1^2\) & \(1^2\) \\
\hline
            \(6\) & \multicolumn{13}{|c|}{\(d=-5\,769\,988\)} \\
\hline
    \(\varkappa\) & \(12\) & \(11\) & \(7\) & \(6\) & \(1\) & \(1\) & \(10\) & \(10\) & \(9\) & \(6\) & \(4\) & \(3\) & \(13\) \\
 \(o(\varkappa)\) & \(2\) & \(0\) & \(1\) & \(1\) & \(0\) & \(2\) & \(1\) & \(0\) & \(1\) & \(2\) & \(1\) & \(1\) & \(1\) \\
         \(\tau\) & \(321^3\) & \(2^21^2\) & \(21^4\) & \(321^3\) & \(2^21^2\) & \(21^4\) & \(21^4\) & \(2^21^2\) & \(2^21^2\) & \(2^21^2\) & \(2^21^2\) & \(21^4\) & \(32^21\) \\
       \(\tau^0\) & \(21\) & \(1^2\) & \(1^2\) & \(21\) & \(1^2\) & \(1^2\) & \(1^2\) & \(1^2\) & \(1^2\) & \(1^2\) & \(1^2\) & \(1^2\) & \(21\) \\
\hline
\end{tabular}
\end{center}
\end{table}



\renewcommand{\arraystretch}{1.1}

\begin{table}[ht]
\caption{Pattern recognition (continued)}
\label{tbl:PatternRecognitionRank3Ctd}
\begin{center}
\begin{tabular}{|r|r|r|r|r|r|r|r|r|r|r|r|r|r|}
\hline
              No. & \multicolumn{13}{|c|}{discriminant} \\
\hline
            \(i\) & \(1\) & \(2\) & \(3\) & \(4\) & \(5\) & \(6\) & \(7\) & \(8\) & \(9\) & \(10\) & \(11\) & \(12\) & \(13\) \\
\hline
            \(7\) & \multicolumn{13}{|c|}{\(d=-7\,060\,148\)} \\
\hline
    \(\varkappa\) & \(2\) & \(4\) & \(4\) & \(9\) & \(4\) & \(10\) & \(8\) & \(10\) & \(10\) & \(1\) & \(6\) & \(8\) & \(3\) \\
 \(o(\varkappa)\) & \(1\) & \(1\) & \(1\) & \(3\) & \(0\) & \(1\) & \(0\) & \(2\) & \(1\) & \(3\) & \(0\) & \(0\) & \(0\) \\
         \(\tau\) & \(21^4\) & \(21^4\) & \(32^21\) & \(2^21^2\) & \(32^21\) & \(21^4\) & \(2^21^2\) & \(321^3\) & \(21^4\) & \(21^4\) & \(2^21^2\) & \(2^21^2\) & \(321^3\) \\
       \(\tau^0\) & \(1^2\) & \(1^2\) & \(21\) & \(1^2\) & \(21\) & \(1^2\) & \(1^2\) & \(21\) & \(1^2\) & \(1^2\) & \(1^2\) & \(1^2\) & \(21\) \\
\hline
            \(8\) & \multicolumn{13}{|c|}{\(d=-8\,180\,671\)} \\
\hline
    \(\varkappa\) & \(12\) & \(9\) & \(2\) & \(6\) & \(10\) & \(6\) & \(8\) & \(2\) & \(10\) & \(10\) & \(9\) & \(11\) & \(4\) \\
 \(o(\varkappa)\) & \(0\) & \(2\) & \(0\) & \(1\) & \(0\) & \(2\) & \(0\) & \(1\) & \(2\) & \(3\) & \(1\) & \(1\) & \(0\) \\
         \(\tau\) & \(321^3\) & \(2^21^2\) & \(21^4\) & \(2^21^2\) & \(2^21^2\) & \(2^21^2\) & \(2^21^2\) & \(21^4\) & \(2^21^2\) & \(2^21^2\) & \(2^21^2\) & \(21^4\) & \(2^21^2\) \\
       \(\tau^0\) & \(21\) & \(1^2\) & \(1^2\) & \(1^2\) & \(1^2\) & \(1^2\) & \(1^2\) & \(1^2\) & \(1^2\) & \(1^2\) & \(1^2\) & \(1^2\) & \(1^2\) \\
\hline
            \(9\) & \multicolumn{13}{|c|}{\(d=-8\,721\,735\)} \\
\hline
    \(\varkappa\) & \(5\) & \(2\) & \(5\) & \(1\) & \(10\) & \(13\) & \(4\) & \(7\) & \(11\) & \(3\) & \(9\) & \(8\) & \(8\) \\
 \(o(\varkappa)\) & \(1\) & \(1\) & \(1\) & \(1\) & \(2\) & \(0\) & \(1\) & \(2\) & \(1\) & \(1\) & \(1\) & \(0\) & \(1\) \\
         \(\tau\) & \(2^21^2\) & \(21^4\) & \(21^4\) & \(21^4\) & \(321^3\) & \(2^21^2\) & \(21^4\) & \(2^21^2\) & \(32^21\) & \(32^21\) & \(21^4\) & \(32^21\) & \(2^21^2\) \\
       \(\tau^0\) & \(1^2\) & \(1^2\) & \(1^2\) & \(1^2\) & \(21\) & \(1^2\) & \(1^2\) & \(1^2\) & \(21\) & \(21\) & \(1^2\) & \(21\) & \(1^2\) \\
\hline
           \(10\) & \multicolumn{13}{|c|}{\(d=-8\,819\,519\)} \\
\hline
    \(\varkappa\) & \(2\) & \(7\) & \(8\) & \(12\) & \(4\) & \(12\) & \(9\) & \(5\) & \(5\) & \(3\) & \(10\) & \(6\) & \(10\) \\
 \(o(\varkappa)\) & \(0\) & \(1\) & \(1\) & \(1\) & \(2\) & \(1\) & \(1\) & \(1\) & \(1\) & \(2\) & \(0\) & \(2\) & \(0\) \\
         \(\tau\) & \(2^21^2\) & \(21^4\) & \(32^21\) & \(2^21^2\) & \(2^21^2\) & \(2^21^2\) & \(2^21^2\) & \(2^21^2\) & \(21^4\) & \(2^21^2\) & \(2^21^2\) & \(2^21^2\) & \(1^6\) \\
       \(\tau^0\) & \(1^2\) & \(1^2\) & \(21\) & \(1^2\) & \(1^2\) & \(1^2\) & \(1^2\) & \(1^2\) & \(1^2\) & \(1^2\) & \(1^2\) & \(1^2\) & \(1^2\) \\
\hline
           \(11\) & \multicolumn{13}{|c|}{\(d=-8\,992\,363\)} \\
\hline
    \(\varkappa\) & \(12\) & \(10\) & \(2\) & \(12\) & \(9\) & \(5\) & \(10\) & \(10\) & \(2\) & \(12\) & \(6\) & \(9\) & \(7\) \\
 \(o(\varkappa)\) & \(0\) & \(2\) & \(0\) & \(0\) & \(1\) & \(1\) & \(1\) & \(0\) & \(2\) & \(3\) & \(0\) & \(3\) & \(0\) \\
         \(\tau\) & \(2^21^2\) & \(21^4\) & \(2^21^2\) & \(2^21^2\) & \(32^21\) & \(2^21^2\) & \(21^4\) & \(21^4\) & \(21^4\) & \(2^21^2\) & \(2^21^2\) & \(21^4\) & \(2^21^2\) \\
       \(\tau^0\) & \(1^2\) & \(1^2\) & \(1^2\) & \(1^2\) & \(21\) & \(1^2\) & \(1^2\) & \(1^2\) & \(1^2\) & \(1^2\) & \(1^2\) & \(1^2\) & \(1^2\) \\
\hline
           \(12\) & \multicolumn{13}{|c|}{\(d=-9\,379\,703\)} \\
\hline
    \(\varkappa\) & \(8\) & \(11\) & \(8\) & \(13\) & \(9\) & \(5\) & \(6\) & \(1\) & \(2\) & \(13\) & \(4\) & \(12\) & \(3\) \\
 \(o(\varkappa)\) & \(1\) & \(1\) & \(1\) & \(1\) & \(1\) & \(1\) & \(0\) & \(2\) & \(1\) & \(0\) & \(1\) & \(1\) & \(2\) \\
         \(\tau\) & \(21^4\) & \(2^21^2\) & \(321^3\) & \(2^21^2\) & \(21^4\) & \(21^4\) & \(2^21^2\) & \(21^4\) & \(21^4\) & \(2^21^2\) & \(2^21^2\) & \(2^21^2\) & \(2^21^2\) \\
       \(\tau^0\) & \(1^2\) & \(1^2\) & \(21\) & \(1^2\) & \(1^2\) & \(1^2\) & \(1^2\) & \(1^2\) & \(1^2\) & \(1^2\) & \(1^2\) & \(1^2\) & \(1^2\) \\
\hline
           \(13\) & \multicolumn{13}{|c|}{\(d=-9\,487\,991\)} \\
\hline
    \(\varkappa\) & \(4\) & \(2\) & \(2\) & \(11\) & \(13\) & \(9\) & \(12\) & \(9\) & \(8\) & \(1\) & \(1\) & \(12\) & \(3\) \\
 \(o(\varkappa)\) & \(2\) & \(2\) & \(1\) & \(1\) & \(0\) & \(0\) & \(0\) & \(1\) & \(2\) & \(0\) & \(1\) & \(2\) & \(1\) \\
         \(\tau\) & \(2^21^2\) & \(2^21^2\) & \(2^21^2\) & \(2^21^2\) & \(321^3\) & \(2^21^2\) & \(2^21^2\) & \(21^4\) & \(2^21^2\) & \(21^4\) & \(2^21^2\) & \(2^21^2\) & \(2^21^2\) \\
       \(\tau^0\) & \(1^2\) & \(1^2\) & \(1^2\) & \(1^2\) & \(21\) & \(1^2\) & \(1^2\) & \(1^2\) & \(1^2\) & \(1^2\) & \(1^2\) & \(1^2\) & \(1^2\) \\
\hline
           \(14\) & \multicolumn{13}{|c|}{\(d=-9\,778\,603\)} \\
\hline
    \(\varkappa\) & \(10\) & \(6\) & \(6\) & \(9\) & \(9\) & \(10\) & \(8\) & \(10\) & \(13\) & \(5\) & \(12\) & \(6\) & \(10\) \\
 \(o(\varkappa)\) & \(0\) & \(0\) & \(0\) & \(0\) & \(1\) & \(3\) & \(0\) & \(1\) & \(2\) & \(4\) & \(0\) & \(1\) & \(1\) \\
         \(\tau\) & \(2^21^2\) & \(321^3\) & \(21^4\) & \(32^21\) & \(32^21\) & \(2^21^2\) & \(2^21^2\) & \(2^21^2\) & \(21^4\) & \(21^4\) & \(2^21^2\) & \(2^21^2\) & \(2^21^2\) \\
       \(\tau^0\) & \(1^2\) & \(21\) & \(1^2\) & \(21\) & \(21\) & \(1^2\) & \(1^2\) & \(1^2\) & \(1^2\) & \(1^2\) & \(1^2\) & \(1^2\) & \(1^2\) \\
\hline
\end{tabular}
\end{center}
\end{table}



\subsection{Arithmetic invariants of fields \(K=\mathbb{Q}(\sqrt{d})\) with \(\mathrm{Cl}_3(K)\simeq (3,3,3)\)}
\label{ss:PatternRecognitionRank3}

After the preliminary data collection in section \S\
\ref{ss:DataCollectionRank3},
we restrict ourselves to the \(14\) cases with elementary abelian \(3\)-class group of type \((3,3,3)\).
The complex quadratic field \(K=\mathbb{Q}(\sqrt{d})\) possesses \(13\) unramified cyclic cubic extensions \(L_1,\ldots,L_{13}\)
with dihedral absolute Galois group \(\mathrm{Gal}(L_i\vert\mathbb{Q})\) of order six
\cite{Ma1}.
Based on Fieker's technique
\cite{Fi},
we use the computational algebra system Magma
\cite{BCFS,MAGMA}
to construct these extensions and to calculate their arithmetical invariants.
In Table
\ref{tbl:PatternRecognitionRank3},
which is continued in Table
\ref{tbl:PatternRecognitionRank3Ctd}
on the following page, we present
the kernel \(\varkappa_i\) of the \(3\)-principalization of \(K\) in \(L_i\)
\cite{Ma,Ma1},
the occupation numbers \(o(\varkappa)_i\) of the principalization kernels
\cite{Ma2}, and
the abelian type invariants \(\tau_i\), resp. \(\tau^0_i\),
of the \(3\)-class group \(\mathrm{Cl}_3(L_i)\), resp. \(\mathrm{Cl}_3(K_i)\),
for each \(1\le i\le 13\)
\cite{Ma3,Ma4}.
Here, we denote by \(K_i\) the unique real non-Galois absolutely cubic subfield of \(L_i\).
For brevity, we give \(3\)-logarithms of abelian type invariants and we denote iteration by formal exponents.
Note that the multiplets \(\varkappa\) and \(\tau\) are \textit{ordered} and in componentwise mutual correspondence,
in the sense of \S\
\ref{s:Correspondence}.



\renewcommand{\arraystretch}{1.1}

\begin{table}[ht]
\caption{Accumulative (unordered) form of IPADs}
\label{tbl:SuccinctData}
\begin{center}
\begin{tabular}{|r|r||c|c|c||c|c||c||l|l|}
\hline
   No.  & discriminant \(d\) & \(2^21^2\) & \(21^4\) & \(1^6\) & \(32^21\) & \(321^3\) & \(431^3\) & polarization & state   \\
\hline
  \(1\) &   \(-4\,447\,704\) &      \(7\) &    \(5\) &   \(0\) &     \(1\) &     \(0\) &     \(0\) & uni          & ground  \\
  \(2\) &   \(-4\,472\,360\) &      \(8\) &    \(4\) &   \(0\) &     \(1\) &     \(0\) &     \(0\) & uni          & ground  \\
  \(3\) &   \(-4\,818\,916\) &      \(8\) &    \(3\) &   \(0\) &     \(1\) &     \(0\) &     \(1\) & bi           & excited \\
  \(4\) &   \(-4\,897\,363\) &      \(8\) &    \(2\) &   \(0\) &     \(1\) &     \(1\) &     \(1\) & tri          & excited \\
  \(5\) &   \(-5\,067\,967\) &      \(7\) &    \(5\) &   \(0\) &     \(1\) &     \(0\) &     \(0\) & uni          & ground  \\
  \(6\) &   \(-5\,769\,988\) &      \(6\) &    \(4\) &   \(0\) &     \(1\) &     \(2\) &     \(0\) & tri          & ground  \\
  \(7\) &   \(-7\,060\,148\) &      \(4\) &    \(5\) &   \(0\) &     \(2\) &     \(2\) &     \(0\) & tetra        & ground  \\
  \(8\) &   \(-8\,180\,671\) &      \(9\) &    \(3\) &   \(0\) &     \(0\) &     \(1\) &     \(0\) & uni          & ground  \\
  \(9\) &   \(-8\,721\,735\) &      \(4\) &    \(5\) &   \(0\) &     \(3\) &     \(1\) &     \(0\) & tetra        & ground  \\
 \(10\) &   \(-8\,819\,519\) &      \(9\) &    \(2\) &   \(1\) &     \(1\) &     \(0\) &     \(0\) & uni          & ground  \\
 \(11\) &   \(-8\,992\,363\) &      \(7\) &    \(5\) &   \(0\) &     \(1\) &     \(0\) &     \(0\) & uni          & ground  \\
 \(12\) &   \(-9\,379\,703\) &      \(7\) &    \(5\) &   \(0\) &     \(0\) &     \(1\) &     \(0\) & uni          & ground  \\
 \(13\) &   \(-9\,487\,991\) &     \(10\) &    \(2\) &   \(0\) &     \(0\) &     \(1\) &     \(0\) & uni          & ground  \\
 \(14\) &   \(-9\,778\,603\) &      \(7\) &    \(3\) &   \(0\) &     \(2\) &     \(1\) &     \(0\) & tri          & ground  \\
\hline
\end{tabular}
\end{center}
\end{table}



In Table
\ref{tbl:SuccinctData},
we classify each of the \(14\) complex quadratic fields \(K=\mathbb{Q}(\sqrt{d})\) of type \((3,3,3)\)
according to the occupation numbers of the abelian type invariants of the \(3\)-class groups \(\mathrm{Cl}_3(L_i)\)
of the \(13\) unramified cyclic cubic extensions \(L_i\),
that is the \textit{accumulated} (unordered) form of the IPAD of \(K\).
Whereas the dominant part of these groups is of order \(3^6=729\),
there always exist(s) at least one and at most four distinguished groups of bigger order,
usually \(3^8=6\,561\) and occasionally even \(3^{10}=59\,049\),
According to the number of distinguished groups,
we speak about \textit{uni-}, \textit{bi-}, \textit{tri-} or \textit{tetra-}polarization.
If the maximal value of the order is \(3^8\), then we have a \textit{ground} state,
otherwise an \textit{excited} state.



\subsection{Proof of Theorem \ref{thm:14NonIsomGrps}}
\label{ss:Proof14NonIsomGrps}

\begin{proof}
According to
\cite[Thm.1.1 and Dfn.1.1, pp.402--403]{Ma4},
the information given in Table
\ref{tbl:SuccinctData}
consists of isomorphism invariants of the metabelian Galois group
\(G=\mathrm{Gal}(\mathrm{F}_3^2(K)\vert K)\)
of the second Hilbert \(3\)-class field of \(K\)
\cite{Ma1}.
Consequently, with respect to the \(13\) abelian type invariants of the \(3\)-class groups \(\mathrm{Cl}_3(L_i)\) alone,
only the groups \(G\) for \(d\in\lbrace -4\,447\,704,-5\,067\,967,-8\,992\,363\rbrace\) could be isomorphic.
However, Tables
\ref{tbl:PatternRecognitionRank3}
and
\ref{tbl:PatternRecognitionRank3Ctd}
show that these three groups differ with respect to another isomorphism invariant,
the \(3\)-principalization type \(\varkappa\)
\cite{Ma,Ma2},
since the corresponding maximal occupation numbers of the multiplet \(o(\varkappa)\) are \(6,2,3\), respectively.
\end{proof}



\subsection{Final remark}
\label{ss:FinalRemark}

We would like to emphasize that Theorem
\ref{thm:14NonIsomGrps}
provides evidence for \textit{a wealth of structure} in the set of infinite \(3\)-class field towers,
which was unknown up to now,
since the common practice is to consider a \(3\)-class field tower as \lq\lq done\rq\rq\
when some criterion in the style of Golod-Shafarevich-Vinberg
\cite{Sh,GoSh,Vb}
or Koch-Venkov
\cite{KoVe}
ensures just its infinity.
However, this perspective is very coarse
and our result proves that it can be refined considerably.

It would be interesting to extend the range of discriminants \(-10^7<d<0\)
and to find the first examples of isomorphic infinite \(3\)-class field towers.

Another very difficult remaining open problem is the actual identification
of the metabelianizations of the \(3\)-tower groups \(G\) of the \(14\) fields.
The complexity of this task is due to unmanageable descendant numbers
of certain vertices, e.g. \(\langle 243,37\rangle\) and \(\langle 729,122\rangle\),
in the tree with root \(\langle 27,5\rangle\).





\begin{thebibliography}{XX}
%
\bibitem{Ar1}
E. Artin,
\textit{Beweis des allgemeinen Reziprozit\"atsgesetzes},
Abh. Math. Sem. Univ. Hamburg
\textbf{5}
(1927),
353--363.
%
\bibitem{Ar2}
E. Artin,
\textit{Idealklassen in Oberk\"orpern und allgemeines Reziprozit\"atsgesetz},
Abh. Math. Sem. Univ. Hamburg
\textbf{7}
(1929),
46--51.
%
\bibitem{BaBu}
L. Bartholdi and M. R. Bush,
\textit{Maximal unramified \(3\)-extensions of imaginary quadratic fields and \(\mathrm{SL}_2\mathbb{Z}_3\)},
J. Number Theory
\textbf{124}
(2007),
159--166.
%
\bibitem{BEO1}
H. U. Besche, B. Eick, and E. A. O'Brien,
\textit{A millennium project: constructing small groups},
Int. J. Algebra Comput.
\textbf{12}
(2002),
623-644.
%
\bibitem{BEO2}
H. U. Besche, B. Eick, and E. A. O'Brien,
\textit{The SmallGroups Library --- a Library of Groups of Small Order},
2005,
an accepted and refereed GAP 4 package, available also in MAGMA.
%
\bibitem{BCP}
W. Bosma, J. Cannon, and C. Playoust,
\textit{The Magma algebra system. I. The user language},
J. Symbolic Comput.
\textbf{24}
(1997),
235--265.
%
\bibitem{BCFS}
W. Bosma, J. J. Cannon, C. Fieker, and A. Steels (eds.),
\textit{Handbook of Magma functions}
(Edition 2.21,
Sydney,
2014).
%
\bibitem{BBH}
N. Boston, M. R. Bush and F. Hajir,
\textit{Heuristics for \(p\)-class towers of imaginary quadratic fields},
to appear in Math. Annalen,
2015.
(arXiv: 1111.4679v1 [math.NT] 20 Nov 2011.)
%
\bibitem{BoLG}
N. Boston, C. Leedham-Green,
\textit{Explicit computation of Galois \(p\)-groups unramified at \(p\)},
J. Algebra
\textbf{256}
(2002),
402--413.
%
\bibitem{BoNo}
N. Boston, H. Nover,
\textit{Computing pro-\(p\) Galois groups},
Proceedings of ANTS 2006, Lecture Notes in Computer Science
\textbf{4076},
1--10,
Springer-Verlag Berlin Heidelberg, 2006.
%
\bibitem{Bl}
D. A. Buell,
\textit{Class groups of quadratic fields},
Math. Comp.
\textbf{30}
(1976),
no. 135,
610--623.
%
\bibitem{Bu}
M. R. Bush,
\textit{Computation of Galois groups associated to the \(2\)-class towers
of some quadratic fields},
J. Number Theory
\textbf{100}
(2003),
313--325.
%
\bibitem{BuMa}
M. R. Bush and D. C. Mayer,
\textit{\(3\)-class field towers of exact length \(3\)},
J. Number Theory
\textbf{147}
(2015),
766--777,
DOI 10.1016/j.jnt.2014.08.010.
(arXiv: 1312.0251v1 [math.NT] 1 Dec 2013.)
%
\bibitem{DD1}
F. Diaz y Diaz,
\textit{Sur les corps quadratiques imaginaires dont le \(3\)-rang du groupe des classes est sup\'erieur \`a \(1\)},
S\'eminaire Delange-Pisot-Poitou,
1973/74,
no. G15.
%
\bibitem{DD2}
F. Diaz y Diaz,
\textit{Sur le \(3\)-rang des corps quadratiques},
No. 78-11,
Univ. Paris-Sud,
1978.
%
\bibitem{Fi}
C. Fieker,
\textit{Computing class fields via the Artin map},
Math. Comp.
\textbf{70}
(2001),
no. 235,
1293--1303.
%
\bibitem{GAP}
The GAP Group,
\textit{GAP -- Groups, Algorithms, and Programming --- a System for Computational Discrete Algebra},
Version 4.7.5,
Aachen, Braunschweig, Fort Collins, St. Andrews,
2014,
\verb+(http://www.gap-system.org)+.
%
\bibitem{GNO}
G. Gamble, W. Nickel, and E. A. O'Brien,
\textit{ANU p-Quotient --- p-Quotient and p-Group Generation Algorithms},
2006,
an accepted GAP 4 package, available also in MAGMA.
%
\bibitem{GoSh}
E. S. Golod and I. R. Shafarevich,
\textit{On class field towers} (Russian),
Izv. Akad. Nauk SSSR, Ser. Mat.
\textbf{28}
(1964),
no. 2,
261--272.
(English transl. in Amer. Math. Soc. Transl. (2)
\textbf{48}
(1965),
91--102.)
%
\bibitem{HeSm}
F.-P. Heider und B. Schmithals,
\textit{Zur Kapitulation der Idealklassen
in unverzweigten primzyklischen Erweiterungen},
J. Reine Angew. Math.
\textbf{336}
(1982),
1--25.
%
\bibitem{HEO}
D. F. Holt, B. Eick, and E. A. O'Brien,
\textit{Handbook of computational group theory},
Discrete mathematics and its applications,
Chapman and Hall/CRC Press,
2005.
%
\bibitem{KoVe}
H. Koch und B. B. Venkov,
\textit{\"Uber den \(p\)-Klassenk\"orperturm eines imagin\"ar-quadratischen Zahlk\"orpers},
Ast\'erisque
\textbf{24--25}
(1975),
57--67.
%
\bibitem{MAGMA}
The MAGMA Group,
\textit{MAGMA Computational Algebra System},
Version 2.21-1,
Sydney,
2014,
\verb+(http://magma.maths.usyd.edu.au)+.
%
\bibitem{Ma}
D. C. Mayer,
\textit{Principalization in complex \(S_3\)-fields},
Congressus Numerantium
\textbf{80}
(1991),
73--87.
(Proceedings of the Twentieth Manitoba Conference on Numerical Mathematics and Computing,
Univ. of Manitoba, Winnipeg, Canada, 1990.)
%
\bibitem{Ma1}
D. C. Mayer,
\textit{The second \(p\)-class group of a number field},
Int. J. Number Theory
\textbf{8}
(2012),
no. 2,
471--505,
DOI 10.1142/S179304211250025X.
%
\bibitem{Ma2}
D. C. Mayer,
\textit{Transfers of metabelian \(p\)-groups},
Monatsh. Math.
\textbf{166}
(2012),
no. 3--4,
467--495,
DOI 10.1007/s00605-010-0277-x.
%
\bibitem{Ma4}
D. C. Mayer,
\textit{The distribution of second \(p\)-class groups on coclass graphs},
J. Th\'eor. Nombres Bordeaux
\textbf{25}
(2013),
no. 2,
401--456,
DOI 10.5802/jtnb842.
(27th Journ\'ees Arithm\'etiques,
Faculty of Mathematics and Informatics,
Univ. of Vilnius,
Lithuania,
2011.)
%
\bibitem{Ma3}
D. C. Mayer,
\textit{Principalization algorithm via class group structure},
J. Th\'eor. Nombres Bordeaux
\textbf{26}
(2014),
no. 2,
415--464.
%
%
\bibitem{Ma6}
D. C. Mayer,
\textit{Periodic bifurcations in descendant trees of finite \(p\)-groups},
to appear in Adv. Pure Math.,
Special Issue on Group Theory,
2015.
\bibitem{Ma7}
D. C. Mayer,
\textit{Index-\(p\) abelianization data of \(p\)-class field tower groups},
29th Journ\'ees Arithm\'etiques,
Univ. of Debrecen,
Hungary,
2015.
%
\bibitem{No}
H. Nover,
Computation of Galois Groups of \(2\)-Class Towers,
Ph.D. Thesis,
University of Wisconsin, Madison,
2009.
%
\bibitem{Ob}
E. A. O'Brien, 
\textit{The p-group generation algorithm}, 
J. Symbolic Comput.
\textbf{9}
(1990),
677--698. 
%
\bibitem{PARI}
The PARI Group, PARI/GP, Version 2.7.2,
Bordeaux,
2014,
\verb+(http://pari.math.u-bordeaux.fr)+.
%
\bibitem{SoTa}
A. Scholz und O. Taussky,
\textit{Die Hauptideale der kubischen Klassenk\"orper
imagin\"ar quadratischer Zahlk\"orper:
ihre rechnerische Bestimmung
und ihr Einflu\ss\ auf den Klassenk\"orperturm},
J. Reine Angew. Math.
\textbf{171}
(1934),
19--41.
%
\bibitem{Sh}
I. R. Shafarevich,
\textit{Extensions with prescribed ramification points} (Russian),
Publ. Math., Inst. Hautes \'Etudes Sci.
\textbf{18}
(1964),
71--95.
(English transl. by J. W. S. Cassels in
Amer. Math. Soc. Transl.,
II. Ser.,
\textbf{59}
(1966),
128--149.)
%
\bibitem{Sk}
D. Shanks,
\textit{Class groups of the quadratic fields found by Diaz y Diaz},
Math. Comp.
\textbf{30}
(1976),
173--178. 
%
\bibitem{OEIS}
N. J. A. Sloane,
\textit{The On-Line Encyclopedia of Integer Sequences} (OEIS),
The OEIS Foundation Inc.,
2014,
\verb+(http://oeis.org/)+.
%
\bibitem{Ta}
O. Taussky,
\textit{A remark concerning Hilbert's Theorem \(94\)},
J. Reine Angew. Math.
\textbf{239/240}
(1970),
435--438.
%
\bibitem{Vb}
E. B. Vinberg,
\textit{On a theorem concerning the infinite-dimensionality of an associative algebra} (Russian),
Izv. Akad. Nauk SSSR, Ser. Mat.
\textbf{29}
(1965),
209--214.
(English transl. in Amer. Math. Soc. Transl. (2)
\textbf{82}
(1969),
237--242.)
%
\end{thebibliography}
\end{document}